\documentclass[10pt,a4paper]{amsart}

\usepackage{fullpage}
\usepackage{amsmath,amsthm,amssymb,mathtools}
\usepackage{enumitem}
\usepackage[colorlinks,citecolor=blue,linkcolor=blue,urlcolor=blue,breaklinks]{hyperref}
\numberwithin{equation}{section}
\numberwithin{figure}{section}

\renewcommand{\leq}{\leqslant}
\newcommand{\Homeo}{\mathrm{Homeo}}

\newtheorem{thm}{Theorem}[section]
\newtheorem{athm}{Theorem}

\newtheorem{acor}[athm]{Corollary}
\newtheorem{lem}[thm]{Lemma}
\newtheorem{prop}[thm]{Proposition}

\newtheorem*{thm*}{Theorem}
\newtheorem*{conj*}{Conjecture}
\newtheorem*{cor*}{Corollary}
\newtheorem*{ques*}{Question}
\newtheorem*{claim*}{Claim}

\theoremstyle{definition}
\newtheorem{rem}[thm]{Remark}
\newtheorem{defn}[thm]{Definition}

\newtheorem{ex}[thm]{Example}

\newtheorem*{rem*}{Remark}

\newcommand{\Cantor}{\mathfrak C}

\DeclareMathOperator{\Aut}{Aut}
\newcommand{\PB}{PB}
\DeclareMathOperator{\gr}{gr}
\hypersetup{pdftitle={Hopficity of braided and labeled Thompson groups},
  pdfsubject={Paper on Hopficity of BV, BF and V(G)}}

\title{Hopficity of braided and labeled Thompson groups}
\date{October 2026}

\author{Mengfei Zhao}
\address{Shanghai Center for Mathematical Sciences, Jiangwan Campus, Fudan University, No.2005 Songhu Road, Shanghai, 200438, P.R. China}
\email{mfzhao22@m.fudan.edu.cn}

\begin{document}
\begin{abstract}
We prove that the braided Thompson groups $BV$ and $BF$ are Hopfian,
answering the two parts of Zaremsky's question. We also prove that, for every discrete group
$G$, the labeled Thompson group $V(G)$ is Hopfian if and only
if $G$ is Hopfian, proving a conjecture of Wu--Wu--Zhao--Zhou. The second result allows us to prove a new embedding result for Hopfian groups: every Hopfian group of type $F_n$ embeds quasi-isometrically as a subgroup of a Hopfian, boundedly acyclic, acyclic group of type $F_n$ that has no proper finite-index subgroups.
\end{abstract}
\maketitle
\enlargethispage{2pt}

\section*{Introduction}
A group $H$ is \emph{Hopfian} if every surjective homomorphism
$\phi:H\to H$ is injective, and hence an automorphism. Equivalently, there
is no nontrivial normal subgroup $N\lhd H$ for which $H/N\cong H$. Thus Hopficity is a rigidity property: a group  cannot be isomorphic to any proper quotient of itself.

The study of this property goes back to Hopf's work on self-maps of surfaces \cite{Hopf1931}.  Hopf proved that every surjective endomorphism of the fundamental group of a closed orientable surface of positive genus is an automorphism. Geometrically, this implies that every continuous self-map of such a surface inducing a surjection on its fundamental group is a homotopy equivalence and, by classical surface rigidity, is homotopic to a homeomorphism. This connection between algebraic and geometric rigidity motivated the study of Hopficity as an important property of groups.

From the definition, it is easy to see that finite groups and simple groups are Hopfian. A classical positive direction is supplied by residual finiteness and finite generation.
A group is \emph{residually finite} if every nontrivial element has
nontrivial image in some finite quotient. Malcev shows that every finitely generated
residually finite group is Hopfian \cite{Malcev1940}. Finite generation alone does not suffice:
Neumann constructed a two-generator group isomorphic to a proper
quotient \cite{Neumann1950}. Higman subsequently gave a finitely
presented example \cite{Higman1951NonHopfian}. Baumslag and Solitar
then exhibited two-generator one-relator non-Hopfian groups
\cite{BaumslagSolitar1962}; a familiar example is
$\operatorname{BS}(2,3)$.
These counterexamples show that even a short finite presentation does
not prevent a group from having a proper self-quotient.

Later developments connected Hopficity with geometric rigidity.
For example, Sela proved that every torsion-free hyperbolic group is
Hopfian \cite{Sela1999}. Such results motivate methods that control
endomorphisms through group structure and geometry, rather than
requiring finite quotients to distinguish elements. The question is
particularly interesting for groups with strong finiteness properties
but few finite quotients: their presentations and classifying spaces
can be well behaved even though the classical residual-finiteness
criterion cannot be applied.

This paper focuses on non-residually finite examples of Hopfian
groups. We study the braided Thompson groups $BV$ and $BF$, and the
labeled Thompson group $V(G)$, where $G$ is an arbitrary
discrete group. For the labeled construction, we determine exactly
which label groups yield Hopfian groups. The proofs use the structure
of the groups and their quotient actions on the kernels. We first
introduce the three families and their distinguishing properties.

\smallskip
\noindent\textbf{The braided Thompson groups $BV$ and $BF$.}
The braided Thompson group $BV$ was introduced independently by
Brin and Dehornoy
\cite{Brin2007,Dehornoy2006}. For the classical Thompson group $V$, one element can be represented by two finite
rooted binary trees with the same number of leaves and a permutation of their leaves. In $BV$, that
permutation is replaced by a braid: the strands record their
crossings as well as their endpoint permutation. Expanding a leaf
splits the corresponding strand into two parallel strands, a process
called \emph{cloning}. Replacing each braid by its endpoint permutation
gives an exact sequence
\[
 1\longrightarrow P\longrightarrow BV\longrightarrow V\longrightarrow1,
\]
where $P$ is the pure braid kernel. Here a \emph{pure braid} is a braid
in which every strand returns to its own labeled endpoint. The group
$P$ is a directed union of finite pure braid groups under cloning, see \cite{Zaremsky2018} for more details.

Recall that, for an integer $n\ge1$, \emph{type $F_n$} means that
a group has a classifying space with finitely many cells in dimensions
at most $n$. A \emph{classifying space} is a connected CW complex
whose fundamental group is the given group and whose universal cover
is contractible. The union of its cells of dimensions at most $n$ is
its \emph{$n$-skeleton}. In particular, type $F_1$ means finite
generation and type $F_2$ means finite presentation.

The group $BV$ is torsion-free and contains an embedded copy of every finite-strand braid group $B_n$, where $n$ is the number of strands \cite{Brin2007}. Brin proved its finite presentation in
\cite{Brin2006}. Moreover, the group $BV$ is also of type $F_\infty$, 
by the theorem of Bux--Fluch--Marschler--Witzel--Zaremsky
\cite{BuxEtAl2016}.
The exact sequence above does not split, since $V$ contains nontrivial
finite-order elements whereas $BV$ is torsion-free
\cite{Brin2007}.
Moreover, every proper normal subgroup of $BV$ lies in $P$, so every
nontrivial quotient surjects onto $V$ and there is no nontrivial finite
quotient \cite[Corollary~2.8]{Zaremsky2018}.
These strong properties coexist with a non-Hopfian pure braid kernel,
as established by Zaremsky
\cite[Proposition~4.8]{Zaremsky2018}.

The braided Thompson group $BF$ is represented by pure braids
between ordered binary trees \cite{BradyBurilloClearyStein2008}.
Forgetting its braid gives a split exact sequence
\[
 1\longrightarrow P\longrightarrow BF\longrightarrow F\longrightarrow1,
 \qquad BF=P\rtimes F,
\]
with the same non-Hopfian pure braid kernel $P$ as for $BV$.
As a subgroup of $BV$, the group $BF$ is torsion-free. It contains
every finite pure braid group and is finitely presented
\cite{BradyBurilloClearyStein2008}. Its type $F_\infty$ property is
proved in \cite{BuxEtAl2016,BuxEtAl2021}. Its abelianization is
$\mathbb Z^4$ \cite[Lemma~1.7]{Zaremsky2018}.
Unlike the quotient $V$ in the previous extension, $F$ preserves the
order of the Cantor space and fixes its two endpoints.

The two braided groups lead to the following question, posed by
Zaremsky in the paper
\cite[Question~4.7]{Zaremsky2018}:
\begin{ques*}[Zaremsky]
Are the braided Thompson groups $BV$ and $BF$ Hopfian?
\end{ques*}
The following
two main theorems answer both parts of this question.

\begin{athm}\label{thm:A}
The braided Thompson group $BV$ is Hopfian.
\end{athm}

\begin{athm}\label{thm:BF-intro}
The braided Thompson group $BF$ is Hopfian.
\end{athm}

Theorems~\ref{thm:A} and~\ref{thm:BF-intro} provide examples in the
class of \emph{non-residually finite, torsion-free Hopfian groups of
type $F_\infty$}. A classical benchmark for this class is Thompson's
group $F$: Brown--Geoghegan established its type $F_\infty$ property
\cite{BrownGeoghegan1984}; its torsion-freeness and the simplicity
of its infinite commutator subgroup are recalled in
\cite{CannonFloydParry1996}. Its Hopficity is also proved below.
The groups $BV$ and $BF$ supply braided examples, containing all finite
braid groups and all finite pure braid groups, respectively. Their
torsion-freeness and type $F_\infty$ were already known from the
original sources cited above; the additional conclusion here is
Hopficity.

\smallskip
\noindent\textbf{The labeled Thompson group $V(G)$.}

The labeled Thompson group $V(G)$ extends Thompson's construction by
allowing labels from a discrete group $G$. It comes with a natural group
embedding $\iota_0:G\hookrightarrow V(G)$, and the assignment
$G\mapsto V(G)$ is a functor from groups to groups.

The following theorem combines the properties established by Thompson, Wu--Wu--Zhao--Zhou and Palmer--Wu \cite{Thompson80,WuWuZhaoZhou2025,PalmerWu2025}.

\begin{thm*}[\cite{Thompson80,WuWuZhaoZhou2025,PalmerWu2025}] For every discrete group $G$, the labeled Thompson group $V(G)$
and the natural embedding $\iota_0$ have the following properties.
\begin{enumerate}[label=(\arabic*),ref=\arabic*,leftmargin=*,itemsep=2pt]
\item\label{item:vg-embedding} The map $\iota_0$ is injective; when $G$
is finitely generated, it is a quasi-isometric embedding.
\item The embedding is Frattini.
\item For finitely generated $G$, $V(G)$
has solvable word problem if and only if $G$ does.
\item\label{item:vg-finiteness} For every integer $n\ge1$,$V(G)$ is of type $F_n$ if and only if $G$ is.
\item\label{item:vg-finite-index} $V(G)$
has no proper finite index subgroup.
\item  $V(G)$ is $5$-uniformly perfect.
\item\label{item:vg-bounded-acyclic} $V(G)$
is boundedly acyclic.
\item $V(G)$ is acyclic.
\end{enumerate}
\end{thm*}

Wu--Wu--Zhao--Zhou formulate the following Hopficity conjecture in
the introduction to their paper
\cite{WuWuZhaoZhou2025}.
\begin{conj*}[Wu--Wu--Zhao--Zhou]
The labeled Thompson group $V(G)$
is Hopfian if and only if $G$ is Hopfian.
\end{conj*}
The next theorem proves this conjecture.

\begin{athm}\label{thm:C}
The labeled Thompson group $V(G)$ is Hopfian if
and only if $G$ is Hopfian.
\end{athm}

Combining Theorem~\ref{thm:C} with the embedding and finiteness
results of Wu--Wu--Zhao--Zhou and Palmer--Wu \cite{WuWuZhaoZhou2025,PalmerWu2025} gives the
following Hopfian refinement of their embedding construction.

\begin{acor}\label{cor:hopfian-embedding}
Let $n\ge1$ be an integer, and let $G$ be a Hopfian group of type $F_n$.
Then $G$ admits a quasi-isometric embedding into a Hopfian, boundedly
acyclic, acyclic group of type $F_n$ having no proper finite
index subgroup.
\end{acor}

The labeled Thompson group $V(G)$ is also non-residually finite,
since it contains the infinite simple group $V$. Unlike the braided
examples above, it contains nontrivial torsion, already in its
label-free copy of $V$.

The two classes of Thompson-type groups considered here also exhibit markedly different bounded
cohomological behavior. Fournier-Facio, Lodha and Zaremsky proved that
$BV$ and $BF$ have infinite-dimensional second bounded cohomology
\cite{FournierFacioLodhaZaremsky2024}, whereas the labeled groups $V(G)$
are boundedly acyclic \cite{WuWuZhaoZhou2025}. Our Hopficity criterion shows
that the labeled construction preserves Hopficity while producing an
acyclic, boundedly acyclic overgroup with the same finiteness properties
of type $F_n$.

It is also worth mentioning that Hopficity of wreath products provides a related perspective.
Bradford and Fournier-Facio
\cite[Corollary~4.3]{BradfordFournierFacio2024}
proved that, for a free abelian group $A$ of positive finite rank
and a finitely generated group $\Gamma$, the restricted wreath product
$A\wr\Gamma$ is Hopfian if and only if $\Gamma$ is Hopfian.
Its kernel $A^{(\Gamma)}$ is nevertheless non-Hopfian whenever
$\Gamma$ is infinite. For general finitely generated abelian base groups,
they connected the corresponding universal Hopficity statement to
Kaplansky's stable finiteness conjecture.
Kochloukova \cite{Kochloukova2026} subsequently established criteria
for permutational wreath products under additional hypotheses.

These works and the present paper study how quotient actions constrain
endomorphisms of non-Hopfian kernels. Their criteria for abelian kernels
use module endomorphism rings and direct finiteness.
Our labeled Thompson construction yields a different criterion:
$V(G)$ is Hopfian precisely when the label group $G$ is Hopfian,
without assuming finite generation or commutativity of $G$.
Its proof uses the local structure of the Cantor action.
For the braided groups, we combine equivariant endomorphism calculations
with the Lie ring structure and residual nilpotence of the pure braid kernel.

\subsection*{AI disclosure.}

Following the work in \cite{WuWuZhaoZhou2025}, the author continued to investigate further properties of $V(G)$, including Hopficity. This research had already yielded partial results before the AI-assisted discussions described below. In September 2026, the author used the OpenAI's ChatGPT model 6.0 Astra while investigating a normal subgroup problem of the labeled Thompson group $V(G)$. Discussions with the model helped the author identify a method for proving injectivity of the restriction of a surjective endomorphism to the kernel $K$, under the assumption that $G$ is Hopfian. The author then observed that the pure braid kernel $P$ in the braided Thompson groups has a similar direct limit structure and proposed extending the approach to this setting. Further AI-assisted discussions contributed to developing a method for proving injectivity of the corresponding restrictions to $P$. All mathematical content was carefully reviewed, verified, and substantially revised by
the author, who takes full responsibility for the final text.

\subsection*{Acknowledgments.}

The author thanks Xiaolei Wu for helpful discussions.

\section{Preliminaries}
\label{sec:cantor-preliminaries}

All groups in this paper are discrete. We use the
commutator convention $[a,b]=aba^{-1}b^{-1}$. For subgroups $A,B$ of a group, $[A,B]$ is the subgroup generated by all commutators $[a,b]$ with $a\in A$ and $b\in B$. For a group $H$, its \emph{abelianization} is
$H_{\mathrm{ab}}=H/[H,H]$, and its center is
$Z(H)=\{z\in H:zh=hz\text{ for every }h\in H\}$.

This section collects the group models and structural facts used by both
braided Thompson groups $BV$, $BF$ and the labeled Thompson group $V(G)$.

\subsection{Cantor partitions and the classical Thompson groups \texorpdfstring{$V$ and $F$}{V and F}}
\label{subsec:cantor-quotients}

\begin{defn}\label{def:cantor-v}
Let $\Cantor=\{0,1\}^{\mathbb N}$, with the product topology obtained from
the discrete topology on $\{0,1\}$. This is a compact Hausdorff space
without isolated points. If $u$ is a finite binary word, the
\emph{cylinder} $[u]=\{u\xi:\xi\in\Cantor\}$ consists of the infinite
words beginning with $u$.
The empty word has cylinder $\Cantor$. All cylinders are clopen, and form a basis of the topology. A \emph{partition} is a finite family of pairwise disjoint cylinders
whose union is $\Cantor$. Thompson's group $V$ consists of all homeomorphisms described by
two partitions $[u_1],\ldots,[u_n]$ and $[w_1],\ldots,[w_n]$, and the
rules
\[
 u_i\xi\longmapsto w_i\xi\qquad(\xi\in\Cantor).
\]
The indexing allows an arbitrary permutation of the cylinders. Subdividing
matched source and target cylinders into their two children leaves the
homeomorphism unchanged.
\end{defn}

Every clopen subset is a finite union of cylinders, by compactness.
A \emph{finite rooted binary tree} is a finite tree with a distinguished
root in which every vertex that is not a leaf has two children, designated
left and right. The geodesics
from the root to a leaf determine 
leaf cylinders which form a partition. Conversely, every partition
arises in this way. Replacing a leaf by two children is called adding a
\emph{caret}; it replaces $[u]$ by $[u0]\sqcup[u1]$, where $\sqcup$
denotes disjoint union. A partition \emph{refines} another if each of its
members is contained in a member of the other. Any two partitions
have a common refinement.

The Thompson group $V$ is simple. Let us recall the following basic facts for $V$. 
 
\begin{lem}\label{lem:partition-matching}
Suppose $(A_1,\ldots,A_r)$ and $(B_1,\ldots,B_r)$ are partitions of $\Cantor$
into nonempty clopen subsets. There is $v\in V$ such that $v(A_i)=B_i$ for
every $i$. Consequently, $V$ is transitive on nonempty proper clopen subsets.
If $U$ is clopen, the analogous assertion holds for partitions of $U$, with
$v$ equal to the identity on $\Cantor\setminus U$.
\end{lem}

For a homeomorphism $v$ of $\Cantor$, its support is the closure of
$\{x:v(x)\neq x\}$. For a clopen subset $U\subseteq\Cantor$, the
\emph{local subgroup}, consisting of elements supported in $U$, is
\[
 V_U=\{v\in V:v|_{\Cantor\setminus U}=\mathrm{id}\}.
\]
If $U\neq\varnothing$, then $V_U$ is isomorphic to $V$.

\begin{lem}\label{lem:local-density}
For every nonempty clopen $U$, the action of $V_U$ on $U$ has dense orbits.
Its diagonal action on ordered pairs of distinct points of $U$ also has dense
orbits. The same holds on unordered pairs of distinct points.
\end{lem}

The classical definition of Thompson group $F$ is described by a subgroup of $Homeo([0,1])$, see \cite{CannonFloydParry1996}. Here, we view $F$ as a subgroup of $V$ consisting of order-preserving Homeomorphisms.

\begin{defn}\label{def:cantor-f}
Give $\Cantor=\{0,1\}^{\mathbb N}$ the \emph{lexicographic order}:
two distinct words are compared at their first differing digit, with
$0<1$. A cylinder is convex for this order. Write $A<B$ for two
disjoint subsets when every point of $A$ precedes every point of $B$.
The subgroup $F\le V$ consists of prefix replacements that preserve
this order. Thus, for ordered partitions
$A_1<\cdots<A_n$ and $B_1<\cdots<B_n$, its tables match $A_i$ with
$B_i$ for each $i$, without permuting indices.
\end{defn}

\begin{lem}\label{lem:f-hopfian}
Thompson's group $F$ is Hopfian.
\end{lem}
\begin{proof}
The classical normal-subgroup theorem says that $F/N$ is abelian
whenever $N\lhd F$ is nontrivial
\cite[Theorem~4.3]{CannonFloydParry1996}. If a surjective
endomorphism $u:F\to F$ had nontrivial kernel, the isomorphism
$F/\ker u\cong F$ would make $F$ abelian. Since $F$ is nonabelian,
$\ker u=\{1\}$.
\end{proof}

\subsection{Rubin's reconstruction}\label{subsec:quotient-spatial-realization}
Rubin’s reconstruction theorem states that every abstract automorphism between locally dense groups of homeomorphisms of locally compact Hausdorff spaces without isolated points is induced by conjugation by a unique homeomorphism between the spaces \cite{Rubin1989}. In this subsection, we give two special cases of this theorem which give us a description of the automorphisms of Thompson group $V$ and $F$.

\begin{thm}[\cite{Rubin1989,BleakCameronMaisselNavasOlukoya2024}]\label{thm:rubin}
Every automorphism $\alpha\in\Aut(V)$ is induced by a homeomorphism
$h\in\Homeo(\Cantor)$:
\[
 \alpha(v)=hvh^{-1}\qquad(v\in V).
\]

\end{thm}

Rubin's theorem also applies to Thompson group $F$ \cite{Brin1996}. In this paper, we define Thompson group $F$ as a subgroup of $V$, hence a subgroup of $Homeo(\mathfrak{C})$. Following Burillo–Cleary \cite{BurilloCleary2013}, we have the following lemma for $F$.

\begin{lem}\label{lem:bf-spatial-f}
For every $\alpha\in\Aut(F)$, there is a homeomorphism
$h\in\Homeo(\Cantor)$, either order-preserving or order-reversing,
such that
\[
 \alpha(f)=hfh^{-1}\qquad(f\in F).
\]
\end{lem}
\begin{proof}
Put $D=\mathbb Z[1/2]=\{a/2^m:a\in\mathbb Z,\ m\ge0\}$.
Following Burillo--Cleary \cite{BurilloCleary2013}, write $\mathrm{PL}_2(\mathbb R)$ for the
group of dyadic piecewise linear homeomorphisms of $\mathbb R$:
the breakpoints form a locally finite subset of $D$, the slopes
belong to $\{\pm2^k:k\in\mathbb Z\}$, and $u(D)=D$ for every
$u\in\mathrm{PL}_2(\mathbb R)$. Both orientations are allowed.
Use the standard increasing change of coordinates
$\rho:(0,1)\to\mathbb R$ from
\cite[Section~1]{BurilloCleary2013}.
Explicitly, prescribe
\[
 \rho(2^{-m})=1-m,\qquad
 \rho(1-2^{-m})=m-1\qquad(m\ge1),
\]
and interpolate affinely between consecutive prescribed points.
Then $\rho(D\cap(0,1))=D$. In this model, Brin's theorem
\cite[Theorem~1(i),(ii)]{Brin1996}, restated in
\cite[Theorem~1.1]{BurilloCleary2013}, realizes $\alpha$
by conjugation with some $u\in\mathrm{PL}_2(\mathbb R)$.
Consequently $h_I=\rho^{-1}u\rho$ realizes $\alpha$ in the interval
model and preserves the interior dyadic points. Extend $h_I$ to
$[0,1]$ by its endpoint limits.

To obtain its Cantor realization, define
\[
 \theta:\Cantor\to[0,1],\qquad
 \theta(x)=\sum_{j\ge1}x_j2^{-j}
 \quad\bigl(x=(x_j)_{j\ge1}\bigr).
\]
The fibers of $\theta$ are singletons except over interior dyadic
points, where they are adjacent pairs in lexicographic order.
Since $h_I$ preserves dyadics, lift it to $h:\Cantor\to\Cantor$
by matching singleton fibers and matching each adjacent pair in the
same order if $u$ is increasing, or in reverse order if $u$ is
decreasing. This is an order isomorphism or order anti-isomorphism,
so it is a homeomorphism: the order topology on $\Cantor$ agrees
with its product topology. By construction $\theta h=h_I\theta$.
The interval and Cantor actions of $F$ are intertwined by $\theta$;
hence $\alpha(f)$ and $hfh^{-1}$ agree on the dense set of singleton
fibers and therefore everywhere, for every $f\in F$.
\end{proof}

\subsection{Braids, the groups \texorpdfstring{$BV$ and $BF$}{BV and BF}, and their common kernel}
\label{subsec:braid-model}

\begin{defn}
For $n\ge2$, the braid group $B_n$ has generators
$\sigma_1,\ldots,\sigma_{n-1}$ and relations
\[
 \sigma_i\sigma_j=\sigma_j\sigma_i\quad(|i-j|\ge2),\qquad
 \sigma_i\sigma_{i+1}\sigma_i=\sigma_{i+1}\sigma_i\sigma_{i+1}.
\]
Geometrically, $\sigma_i$ is the positive half twist exchanging the
$i^{th}$ and $(i+1)^{st}$ strands, numbered from left to right. The group
$\mathfrak S_n$ is the group of permutations of $\{1,\ldots,n\}$.
The endpoint permutation homomorphism $B_n\to\mathfrak S_n$ sends
$\sigma_i$ to the transposition $(i\ i+1)$. Its kernel
$\PB_n=\ker(B_n\to\mathfrak S_n)$ is the \emph{pure braid group}:
each strand ends at the same labeled position at which it starts.
We set $B_1=\PB_1=1$.
\end{defn}

The standard pure generators may be chosen as
\[
 A_{ij}=\sigma_{j-1}\cdots\sigma_{i+1}\sigma_i^2
              \sigma_{i+1}^{-1}\cdots\sigma_{j-1}^{-1}
 \qquad(1\le i<j\le n).
\]
For $i<j$, delete every strand except the $i^{th}$ and $j^{th}$. This gives
a homomorphism $\PB_n\to\PB_2=\langle\sigma_1^2\rangle\cong\mathbb Z$.
The resulting exponent is the \emph{pairwise linking number}
$\ell_{ij}(p)\in\mathbb Z$ of a pure braid $p$. The normalization on generators gives
$\ell_{ij}(A_{rs})=1$ if $(i,j)=(r,s)$ and $0$ otherwise. Write $t_{ij}=A_{ij}[\PB_n,\PB_n]$ for the abelianization class of the
standard generator, and $[p]=p[\PB_n,\PB_n]$ for the class of $p$.
The standard abelianization calculation gives
\begin{equation}\label{eq:pb-abelianization}
 (\PB_n)_{\mathrm{ab}}
   =\bigoplus_{1\le i<j\le n}\mathbb Zt_{ij},
 \qquad [p]=\sum_{i<j}\ell_{ij}(p)t_{ij}.
\end{equation}
See \cite{CohenPrassidis2006} for more information on the generators and pairwise linking numbers of pure braid groups.

\begin{defn}[The braided Thompson groups $BV$ and $BF$]
\label{def:bv}\label{def:bf-model}
A \emph{braided tree-pair diagram} is a triple $(T_-,b,T_+)$, where $T_-$
and $T_+$ are finite rooted binary trees with the same number $n$ of
leaves and $b\in B_n$ joins their leaves. An expansion adds a caret at a
leaf, replaces its braid strand by two parallel strands, and adds the
matching caret at the other end of that strand. Diagrams are identified
under expansions and their inverses. Their equivalence classes, denoted $[T_-,b,T_+]$, form the braided Thompson group $ BV$.
Multiplication is given by stacking diagrams after refining the middle trees
to agree; the inverse reverses the diagram and inverts the braid.

The \emph{braided Thompson group $BF$} is the subgroup of $BV$ whose
diagrams have pure braids. Cloning and its inverse preserve purity, so
this condition is independent of the representative. Stacking pure
braids and inverting them preserve purity, so these classes form a
subgroup. Forgetting the pure braid defines a surjective homomorphism
$\pi_F:BF\to F$.
\end{defn}

This is the construction of the braided Thompson group $BV$
\cite{Brin2007,Dehornoy2006}; the
tree-pair model and pure variants are developed in
\cite{BradyBurilloClearyStein2008}. Replacing a braid by its endpoint
permutation defines a surjective homomorphism
\begin{equation}\label{eq:bv-extension}
 \pi: BV\longrightarrow V.
\end{equation}
Let $P=\ker\pi$. An element of $P$ has a representative $[T,p,T]$ with
$p\in\PB_n$. Conversely, all such representatives belong to $P$.
For a tree $T$ with $n$ leaves, write
$\PB_T=\{[T,p,T]:p\in\PB_n\}$ for this copy of $\PB_n$ in $P$.
Write $T\leq U$ if $U$ is obtained from $T$ by adding carets.
For such a refinement, cloning defines the inclusion
$i_{T,U}:\PB_T\to\PB_U$. So this forms a directed system.
Cloning is injective, since deleting one of the two clones recovers the
original braid. Consequently
\begin{equation}\label{eq:p-colimit}
 P=\varinjlim_T\PB_T=\bigcup_T\PB_T,
\end{equation}
with respect to common tree refinements; $\varinjlim$ denotes this direct
limit. Concretely, every element belongs to a finite stage $\PB_T$, and
every finite collection of elements belongs to a common refined stage.
This description is also given in \cite[Section~1.2]{Zaremsky2018}.

The kernel of $\pi_F$ is the same group $P=\ker\pi$: a diagram
projects to the identity of $F$ precisely when it can be represented as
$[T,p,T]$ with $p$ pure. Thus the two braided groups have the common
kernel in the exact sequences
\[
 1\longrightarrow P\longrightarrow BV\xrightarrow{\pi}V
 \longrightarrow1
\]
and
\begin{equation}\label{eq:bf-extension}
 1\longrightarrow P\longrightarrow BF\xrightarrow{\pi_F}F
 \longrightarrow1,\qquad BF=P\rtimes F.
\end{equation}
For the second sequence, send an ordered tree pair to its diagram with
trivial braid. This gives a homomorphism $s:F\to BF$ with
$\pi_Fs=\operatorname{id}_F$, and hence the displayed splitting.
The action is $f\cdot p=s(f)ps(f)^{-1}$. With the pair $(p,f)$
representing $ps(f)$, multiplication is
\[
 (p,f)(q,g)=\bigl(p\,(f\cdot q),fg\bigr).
\]

This is the splitting in \cite[Section~1.2]{Zaremsky2018}. Both groups therefore use exactly the same finite stages $\PB_T$ and
cloning maps. Unlike cloning on braid
groups with unrestricted endpoint permutations, cloning a specified pure
strand is a group homomorphism. Moreover, it has the following retraction which will be used for the
residual-nilpotence argument.

\begin{lem}\label{lem:split-cloning}
If $T\leq U$, the cloning inclusion $i_{T,U}:\PB_T\to\PB_U$ has a
homomorphic left inverse $d_{U,T}:\PB_U\to\PB_T$, meaning
$d_{U,T}\circ i_{T,U}=\operatorname{id}_{\PB_T}$. The map $d_{U,T}$ is
obtained by retaining exactly one descendant strand of each leaf of $T$
and deleting all other strands.
\end{lem}
\begin{proof}
Pure braids preserve each strand label, so deletion of a fixed set of strands
commutes with multiplication. The retained strand in each cloned bundle
follows the original strand, and deleting its parallel companions recovers
the original pure braid. This proves the left-inverse identity.
\end{proof}

\subsection{Normal subgroups and a common epimorphism reduction}
\label{subsec:braided-epi-reduction}
Let us recall the normal subgroup structure for $BV$ and $BF$ as in \cite[Corollary~2.8]{Zaremsky2018} and
\cite[Theorem~2.1]{Zaremsky2018}.
\begin{thm}
\label{thm:braided-normal-subgroups}

\begin{enumerate}[label=(\roman*)]
\item Every proper normal subgroup of $BV$ is contained in $P$.
Consequently, $P$ is the unique maximal proper normal subgroup of $BV$.
\item For every normal subgroup $N\lhd BF$, either $N\subseteq P$
or $[BF,BF]\subseteq N$.
\end{enumerate}
\end{thm}

 The two normal-subgroup results have different conclusions, but they
give the same reduction for surjective endomorphisms.

\begin{lem}
\label{lem:braided-epi-reduction}
Let $(E,Q,\pi_E)$ be either $(BV,V,\pi)$ or $(BF,F,\pi_F)$.
For every surjective endomorphism $\Phi:E\twoheadrightarrow E$,
\[
 \Phi^{-1}(P)=P,\qquad \Phi(P)=P,\qquad \ker\Phi\subseteq P.
\]
It induces an automorphism $\alpha:Q\to Q$ characterized by
$\pi_E\Phi=\alpha\pi_E$, and its restriction
$\beta=\Phi|_P:P\twoheadrightarrow P$ is surjective. Moreover,
$\ker\Phi=\ker\beta$, so $\Phi$ is injective if and only if
$\beta$ is injective.
\end{lem}
\begin{proof}
Let $q=\pi_E \circ \Phi:E\twoheadrightarrow Q$ and $N=\ker q$.
Then $E/N\cong Q$. If $E=BV$, the subgroup $N$ is proper, so
Theorem~\ref{thm:braided-normal-subgroups}(i) gives $N\subseteq P$.
If $E=BF$ and $N\not\subseteq P$, part (ii) would give
$[BF,BF]\subseteq N$, making $BF/N\cong F$ abelian. This contradicts
the nonabelianity of $F$. Thus $N\subseteq P$ in both cases.

Because $N\subseteq\ker\pi_E$, the formula
\[
 u:Q\longrightarrow Q,\qquad u(q(g))=\pi_E(g)\quad(g\in E)
\]
is well-defined. Indeed, if $q(g)=q(g')$, then
$g^{-1}g'\in N\subseteq P$, so $\pi_E(g)=\pi_E(g')$.
The formula defines a homomorphism, and it is surjective because
$\pi_E$ is surjective. The group $Q$ is Hopfian, by the simplicity of $V$ and Lemma~\ref{lem:f-hopfian}
for $F$. Hence $u$ is an automorphism. The identity $\pi_E=uq$ yields
\[
 P=\ker\pi_E=\ker q=\Phi^{-1}(P).
\]
In particular $\Phi(P)\subseteq P$. For $p\in P$, choose $g\in E$
with $\Phi(g)=p$. The preimage equality puts $g$ in $P$, proving
$\Phi(P)=P$ and surjectivity of $\beta$.

Taking $\alpha=u^{-1}$ gives $\pi_E\Phi=\alpha\pi_E$.
Finally $\ker\Phi\subseteq\Phi^{-1}(P)=P$, so
$\ker\Phi=\ker\beta$, as claimed.
\end{proof}

\section{Residual nilpotence and graded Lie rings}
\label{sec:graded-preliminaries}

 In this section, we collect the general tools used for $BV$ and $BF$ in
Sections~\ref{sec:bv-hopficity} and~\ref{sec:bf-hopficity}. Following Lemma~\ref{lem:braided-epi-reduction}, to get the injectivity of $\Phi$, it suffices to get the injectivity of $\beta$. We
prove that the common pure braid kernel
$P$ is residually nilpotent. For residually nilpotent group, we can use the graded Lie ring method to get the injectivity of the group homomorphism starting from a residually nilpotent group. We establish the graded injectivity criterion, using the idea in \cite{CohenPrassidis2006}
and \cite{CohenWu2008}.

\begin{defn}
Let $H$ be a group. The
\emph{lower central series} of $H$ is the sequence of characteristic subgroups, which are in particular normal, given by
\[
 \gamma_1(H)=H,\qquad
 \gamma_{k+1}(H)=[H,\gamma_k(H)]\quad(k\ge1).
\]
Thus $\gamma_2(H)=[H,H]$ is the commutator subgroup. The sequence is
descending: $\gamma_{k+1}(H)\subseteq\gamma_k(H)$ for all $k$.

The group $H$ is \emph{nilpotent of class at most $c$} if
$\gamma_{c+1}(H)=\{1\}$. It is \emph{residually nilpotent} if for every
$x\in H\setminus\{1\}$ there exist a nilpotent group $N$ and a
homomorphism $q:H\to N$ such that $q(x)\neq1$.
\end{defn}

The lower central series can be defined for all groups, and they are preserved by group homomorphism. One can also use the lower central series to determine the residual nilpotence of groups.

\begin{lem}\label{lem:lcs-functorial}
If $f:H\to K$ is a group homomorphism, then
\[
 f(\gamma_k(H))\subseteq\gamma_k(K)\qquad(k\ge1).
\]
Consequently, $H$ is residually nilpotent if and only if
\begin{equation}\label{eq:residually-nilpotent-intersection}
 \bigcap_{k\ge1}\gamma_k(H)=\{1\}.
\end{equation}
\end{lem}
\begin{proof}
For $k=1$, the claimed inclusion is immediate. If it holds for $k$, the
identity $f([a,b])=[f(a),f(b)]$ gives
\[
 f(\gamma_{k+1}(H))
 =f([H,\gamma_k(H)])
 \subseteq[K,\gamma_k(K)]=\gamma_{k+1}(K).
\]
This proves the inclusion by induction.

Suppose the intersection in
\eqref{eq:residually-nilpotent-intersection} is trivial. Given $x\neq1$,
choose $k$ such that $x\notin\gamma_k(H)$. The quotient map
$H\to H/\gamma_k(H)$ detects $x$, and its target is nilpotent of class
at most $k-1$. Conversely, if $x$ lies in every $\gamma_k(H)$ and
$q:H\to N$ is a homomorphism to a nilpotent group of class at most $c$,
then
\[
 q(x)\in q(\gamma_{c+1}(H))
 \subseteq\gamma_{c+1}(N)=\{1\}.
\]
If $H$ is residually nilpotent, this forces $x=1$.
\end{proof}

\begin{lem}\label{lem:commutator-weights}
For all positive integers $i,j$,
\begin{equation}\label{eq:commutator-weights}
 [\gamma_i(H),\gamma_j(H)]\subseteq\gamma_{i+j}(H).
\end{equation}
\end{lem}
\begin{proof}
Denote ${}^g u=gug^{-1}$ by the conjugation. Then we have the Hall--Witt identity:
\begin{equation}\label{eq:hall-witt}
 {}^y[[y^{-1},x],z]\;
 {}^z[[z^{-1},y],x]\;
 {}^x[[x^{-1},z],y]=1.
\end{equation}

As a consequence of the identity: for subgroups $A,B,C$ and a normal subgroup $D$, the inclusions
$[[B,C],A]\subseteq D$ and $[[C,A],B]\subseteq D$ imply
$[[A,B],C]\subseteq D$.

We now induct on $i$, proving the assertion for all $j$ simultaneously.
For $i=1$, it is the definition of the lower central series. Suppose
$i\ge2$ and the assertion is known with $i-1$ in place of $i$.
Apply the three-subgroup consequence with
\[
 A=H,\quad B=\gamma_{i-1}(H),\quad C=\gamma_j(H),
 \quad D=\gamma_{i+j}(H).
\]
The induction hypothesis gives
\[
 [[B,C],A]\subseteq[\gamma_{i+j-1}(H),H]
 =\gamma_{i+j}(H)
\]
and
\[
 [[C,A],B]=[\gamma_{j+1}(H),\gamma_{i-1}(H)]
 \subseteq\gamma_{i+j}(H).
\]
The conclusion is
$[\gamma_i(H),\gamma_j(H)]\subseteq\gamma_{i+j}(H)$, as required.
\end{proof}

With lower central series, we can get a graded ring with Lie bracket as follows:

\begin{defn}\label{def:graded-cosets}
For $k\ge1$, define the $k^{th}$ \emph{graded piece} by
\[
 \gr_k(H)=\gamma_k(H)/\gamma_{k+1}(H).
\]
If $x\in\gamma_k(H)$, its class in this quotient is the coset
\[
 \overline{x}_k=x\gamma_{k+1}(H)
 =\{xu:u\in\gamma_{k+1}(H)\}.
\]
The canonical quotient homomorphism is
\[
 q_k:\gamma_k(H)\longrightarrow\gr_k(H),
 \qquad q_k(x)=\overline{x}_k.
\]
By Lemma~\ref{lem:commutator-weights},
$[\gamma_k(H),\gamma_k(H)]\subseteq\gamma_{2k}(H)
\subseteq\gamma_{k+1}(H)$, so this quotient is abelian. We therefore
write its operation additively:
\[
 \overline{x}_k+\overline{y}_k=\overline{xy}_k,
 \qquad -\overline{x}_k=\overline{x^{-1}}_k,
 \qquad 0=1\gamma_{k+1}(H).
\]
In particular,
\begin{equation}\label{eq:zero-graded-coset}
 \overline{x}_k=0\quad\Longleftrightarrow\quad
 x\in\gamma_{k+1}(H).
\end{equation}
The notation $\overline{x}_k$ is used only when $x\in\gamma_k(H)$.
\end{defn}

\begin{prop}\label{prop:graded-lie-ring}
The direct sum
\[
 \gr(H)=\bigoplus_{k\ge1}\gr_k(H)
\]
is a graded Lie ring, with bracket on homogeneous elements defined by
\begin{equation}\label{eq:graded-bracket}
 [\overline{x}_i,\overline{y}_j]
 =\overline{[x,y]}_{i+j}.
\end{equation}
Here a \emph{homogeneous} element of degree $k$ means an element of
$\gr_k(H)$. A \emph{Lie ring} is an abelian group with a bilinear
bracket satisfying $[u,u]=0$ and the Jacobi identity
\[
 [[u,v],w]+[[v,w],u]+[[w,u],v]=0.
\]
The bracket is graded in the sense that
$[\gr_i(H),\gr_j(H)]\subseteq\gr_{i+j}(H)$. It satisfies the
ordinary Lie identities, with no degree-dependent sign convention.
\end{prop}

\begin{proof}
Equation~\eqref{eq:commutator-weights} places $[x,y]$ in the required
subgroup. To prove independence of representatives, replace
$x\in\gamma_i(H)$ by $xu$, where $u\in\gamma_{i+1}(H)$.
The group commutator identity
\begin{equation}\label{eq:commutator-product-first}
 [xu,y]=x[u,y]x^{-1}[x,y]
\end{equation}
and the inclusion $[u,y]\in\gamma_{i+j+1}(H)$ show that the resulting
class in $\gr_{i+j}(H)$ does not change. The identity
\begin{equation}\label{eq:commutator-product-second}
 [x,yu]=[x,y]y[x,u]y^{-1}
\end{equation}
gives the same conclusion for the second representative.

If $t\in\gamma_s(H)$ and $g\in H$, then
$gtg^{-1}t^{-1}=[g,t]\in\gamma_{s+1}(H)$. Thus conjugation by any
element of $H$ acts trivially on $\gr_s(H)$. Applying this observation
to \eqref{eq:commutator-product-first} and
\eqref{eq:commutator-product-second}, with factors in the same
respective degrees, proves additivity in each argument. We extend the
bracket to $\gr(H)$ by bilinearity. This involves only finite sums,
because an element of a direct sum has only finitely many nonzero
components.

The identities $[x,x]=1$ and $[y,x]=[x,y]^{-1}$ give alternation on
homogeneous elements and skew-symmetry. These imply alternation on
arbitrary finite sums as well. Finally, put
$x\in\gamma_i(H)$, $y\in\gamma_j(H)$, and
$z\in\gamma_\ell(H)$ into \eqref{eq:hall-witt} and reduce modulo
$\gamma_{i+j+\ell+1}(H)$. All conjugations disappear in that
quotient. The class of $[y^{-1},x]$ is
$[-\overline y_j,\overline x_i]
=[\overline x_i,\overline y_j]$, and the other two inner
commutators behave cyclically. The product in the Hall--Witt identity
therefore becomes the sum
\[
 [[\overline x_i,\overline y_j],\overline z_\ell]
 +[[\overline y_j,\overline z_\ell],\overline x_i]
 +[[\overline z_\ell,\overline x_i],\overline y_j]=0.
\]
This is Jacobi on homogeneous elements; bilinearity gives the general
identity.
\end{proof}

\begin{prop}
\label{prop:induced-graded-map}
Every group homomorphism $f:H\to K$ induces homomorphisms of abelian
groups
\begin{align*}
 \gr_k(f):\gr_k(H)&\longrightarrow\gr_k(K),\\
 x\gamma_{k+1}(H)&\longmapsto f(x)\gamma_{k+1}(K)
 \qquad(k\ge1).
\end{align*}
Their direct sum
\[
 \gr(f)=\bigoplus_{k\ge1}\gr_k(f):\gr(H)\longrightarrow\gr(K)
\]
is a degree-preserving Lie ring homomorphism.
\end{prop}
\begin{proof}
Lemma~\ref{lem:lcs-functorial} shows that $f(x)$ belongs to the
required subgroup. If $x\gamma_{k+1}(H)=x'\gamma_{k+1}(H)$, then
$x^{-1}x'\in\gamma_{k+1}(H)$, so
$f(x)^{-1}f(x')\in\gamma_{k+1}(K)$. Thus the formula is
well-defined. It is additive because $f(xx')=f(x)f(x')$.
For $x\in\gamma_i(H)$ and $y\in\gamma_j(H)$,
\begin{align*}
 \gr_{i+j}(f)([\overline x_i,\overline y_j])
 &=\overline{f([x,y])}_{i+j}\\
 &=\overline{[f(x),f(y)]}_{i+j}\\
 &=[\gr_i(f)(\overline x_i),\gr_j(f)(\overline y_j)].
\end{align*}
The overlines on the right-hand sides denote classes in the quotients
of $K$. This is bracket preservation; bilinearity extends it to all
of $\gr(H)$.

The formulas also show that $\gr(\operatorname{id}_H)$ is the identity
and $\gr(g\circ f)=\gr(g)\circ\gr(f)$ for a further homomorphism
$g:K\to L$ into a group $L$. In particular, a group automorphism induces a graded Lie
ring automorphism, its inverse being induced by the inverse group
automorphism.
\end{proof}

This construction for filtration-preserving homomorphisms is
\cite[Theorem~(2.2)]{Lazard1954}. The lower central
series is preserved by every group homomorphism by
Lemma~\ref{lem:lcs-functorial}.

\begin{lem}
\label{lem:degree-one}
For every group $H$, the Lie ring $\gr(H)$ is generated by
$\gr_1(H)=H_{\mathrm{ab}}$. More explicitly, every
$z\in\gr_k(H)$ can be written
\begin{equation}\label{eq:finite-lie-polynomial}
 z=F(m_1,\ldots,m_r),\qquad m_1,\ldots,m_r\in H_{\mathrm{ab}},
 \quad r<\infty,
\end{equation}
where $F$ is a finite integral linear combination of iterated Lie
brackets, each containing exactly $k$ entries. We call such an $F$ a
\emph{homogeneous Lie polynomial of degree $k$}.
\end{lem}
\begin{proof}
Induct on $k$. For $k=1$, the assertion follows from
$\gr_1(H)=H/\gamma_2(H)$. For $k\ge2$, the definition
$\gamma_k(H)=[H,\gamma_{k-1}(H)]$ means that every representative of
$z$ is a finite product of commutators $[a,b]$ and their inverses, where
$a\in H$ and $b\in\gamma_{k-1}(H)$. Passing to $\gr_k(H)$ turns
this finite product into a finite sum of terms
$\pm[\overline a_1,\overline b_{k-1}]$.
By induction each $\overline b_{k-1}$ is a finite integral sum of
brackets with $k-1$ degree-one entries. Bilinearity therefore expresses
$z$ as a finite integral sum of brackets with $k$ degree-one entries.
Only finitely many entries occur in this finite expression, and they
can be listed as $m_1,\ldots,m_r$.
\end{proof}

Degree-one generation is also stated in
\cite[Chapter~I]{Lazard1954}.

The expression in
\eqref{eq:finite-lie-polynomial} need not be unique; the lemma does not
assert that $\gr(H)$ is a free Lie ring. Every graded Lie homomorphism
preserves evaluation of these polynomials. Moreover, for an integer
$a$ and a homogeneous polynomial $F$ of degree $k$, bilinearity gives
\begin{equation}\label{eq:homogeneous-scaling}
 F(am_1,\ldots,am_r)=a^kF(m_1,\ldots,m_r).
\end{equation}

\begin{thm}
\label{thm:graded-detection}
Let $H$ be residually nilpotent and let $f:H\to K$ be a group
homomorphism. If $\gr_k(f)$ is injective for every $k\ge1$, then $f$
is injective.
\end{thm}
\begin{proof}
Let $x\in\ker f$, so $f(x)=1$. We prove by induction that
$x\in\gamma_k(H)$ for every $k$. The assertion for $k=1$ is automatic.
Suppose $x\in\gamma_k(H)$. Its class
$\overline x_k=x\gamma_{k+1}(H)$ is therefore defined, and
\[
 \gr_k(f)(\overline x_k)
 =f(x)\gamma_{k+1}(K)
 =1\gamma_{k+1}(K)=0.
\]
Injectivity of $\gr_k(f)$ implies $\overline x_k=0$.
Equation~\eqref{eq:zero-graded-coset} then gives
$x\in\gamma_{k+1}(H)$, completing the induction. Residual nilpotence
now yields
\[
 x\in\bigcap_{k\ge1}\gamma_k(H)=\{1\}.
\]
Thus $\ker f=\{1\}$.
\end{proof}

For a degree-preserving map on a direct sum, injectivity is equivalent
to injectivity on every homogeneous piece. Hence the hypothesis can
also be stated as injectivity of the full Lie ring map $\gr(f)$.
This is exactly \cite[Lemma~3.1(1)]{CohenWu2008} and is also used in
\cite[Section~1]{CohenPrassidis2006}; we have included
its complete proof. Injectivity on degree one \emph{alone} does not
suffice: the abelianization map from a free group of rank two to
$\mathbb Z^2$ is an isomorphism on degree one but kills its nontrivial
commutator subgroup. The finite-realization argument in Lemma~\ref{lem:finite-realization}
will supply injectivity in every degree for the application to $P$.

We finish this section by proving residual nilpotence of the common
pure braid kernel $P$ of $BV$ and $BF$, defined in
Subsection~\ref{subsec:braid-model}. We first recall the finite-strand
result and then use the split cloning maps to pass to $P$.

\begin{lem}\label{thm:pb-residual}
For every positive integer $n$, the pure braid group $\PB_n$ is residually nilpotent.
\end{lem}
\begin{proof}
For $n\ge2$, forgetting the last strand gives the split
Fadell--Neuwirth sequence \cite{FadellNeuwirth1962}
\[
 1\longrightarrow F_{n-1}\longrightarrow\PB_n
 \longrightarrow\PB_{n-1}\longrightarrow1.
\]
Here $F_{n-1}$ is a free group of rank $n-1$, identified with the
fundamental group of the plane punctured at the first $n-1$ marked
points. To choose a splitting, realize the first $n-1$ strands inside
a disk and add a stationary last strand outside that disk. The action
of $\PB_{n-1}$ on the normal subgroup $F_{n-1}$ fixes its
abelianization: each meridian of a puncture is sent to a conjugate of
itself. The splitting therefore writes $\PB_n$ as an
almost-direct product $F_{n-1}\rtimes\PB_{n-1}$, meaning a
semidirect product in which the second factor acts trivially on the
abelianization of the first factor.

The Falk--Randell lower-central-series lemma for almost-direct
products \cite{FalkRandell1985,FalkRandell1988} gives
\begin{equation}\label{eq:pb-lower-central-splitting}
 \gamma_k(\PB_n)
 =\gamma_k(F_{n-1})\rtimes\gamma_k(\PB_{n-1})
 \qquad(k\ge1).
\end{equation}
We use the classical free-group facts that free groups are residually nilpotent \cite{Magnus1935}. Together with
\eqref{eq:pb-lower-central-splitting}, these are the standard
ingredients of the pure-braid residual-nilpotence argument in
\cite{FalkRandell1985,FalkRandell1988}; see also
\cite[Section~1]{Marin2012}.

Induct on $n$, starting with $\PB_1=\{1\}$. If
$p\in\bigcap_k\gamma_k(\PB_n)$, its image in $\PB_{n-1}$ lies in
every lower central subgroup, and is therefore trivial by induction.
Thus $p\in F_{n-1}$. Equation~\eqref{eq:pb-lower-central-splitting}
places $p$ in every $\gamma_k(F_{n-1})$, so $p=1$.

\end{proof}

\begin{lem}\label{lem:p-residual}
For every finite rooted binary tree $T$ and every $k\ge1$,
\begin{equation}\label{eq:p-stage-intersection}
 \gamma_k(P)\cap\PB_T=\gamma_k(\PB_T).
\end{equation}
Consequently, $P$ is residually nilpotent.
\end{lem}
\begin{proof}
Recall that $\PB_T$ is the pure braid group indexed by the leaves of
$T$, viewed as a subgroup of $P$. The refinement relation $U\ge T$
means that $U$ is obtained from $T$ by repeatedly splitting leaves.
For such a refinement, Lemma~\ref{lem:split-cloning} supplies the
cloning inclusion $i_{T,U}:\PB_T\hookrightarrow\PB_U$ and a
strand-deletion homomorphism $d_{U,T}:\PB_U\to\PB_T$ with
$d_{U,T}i_{T,U}=\operatorname{id}_{\PB_T}$. Retaining one chosen
descendant of each leaf of $T$ constructs this left inverse.

Functoriality gives the inclusion from right to left in
\eqref{eq:p-stage-intersection}. Conversely, let
$p\in\gamma_k(P)\cap\PB_T$. Membership in $\gamma_k(P)$ is
witnessed by a finite expression: recursively use
$\gamma_{s+1}(P)=[P,\gamma_s(P)]$ to write $p$ as a finite product
of commutators whose entries are themselves finite products of
commutators of the required smaller weights. Only finitely many
elements of $P$ occur in this expression. Choose a finite tree stage
for each of them, and then choose a common refinement $U$ of those
trees and $T$. All entries now belong to $\PB_U$. The same expression
shows that $p\in\gamma_k(\PB_U)$; equality of the expression and
$p$ holds there because the stage embeddings into $P$ are injective.
Applying the homomorphism $d_{U,T}$ gives
\[
 p=d_{U,T}(p)\in d_{U,T}(\gamma_k(\PB_U))
 \subseteq\gamma_k(\PB_T),
\]
where $p$ on the left is its original representative in $\PB_T$.
This proves \eqref{eq:p-stage-intersection}.

Finally, if $p\in\bigcap_k\gamma_k(P)$, choose one tree $T$ with
$p\in\PB_T$. The equality just proved implies
\[
 p\in\bigcap_{k\ge1}\gamma_k(\PB_T)=\{1\},
\]
where the final equality is Lemma~\ref{thm:pb-residual}.
Thus $P$ is residually nilpotent.
\end{proof}

The strand-deletion retractions in Lemma~\ref{lem:split-cloning}
ensure that no finite-stage element enters a lower central subgroup
only after passing to a larger tree stage. They are the reason the
finite pure-braid intersection calculation also proves residual
nilpotence of $P$.

Note that the kernel $P$ is not finitely generated. So the residual nilpotence does not imply the Hopficity. In fact, the kernel $P$ is not Hopfian \cite{Zaremsky2018}.

\begin{prop}\cite[Proposition~4.8]{Zaremsky2018}
\label{prop:p-nonhopfian}
The common pure braid kernel
\[
 P=\ker(BV\to V)=\ker(BF\to F)
\]
is non-Hopfian. More precisely, it admits a surjective, noninjective
endomorphism obtained by deleting all strands in the right half of a
tree-braid-tree diagram.
\end{prop}

\section{Hopficity of the braided Thompson group \texorpdfstring{$BV$}{BV}}
\label{sec:bv-hopficity}

The goal of this section is to prove that $BV$ is Hopfian, by showing
that every surjective endomorphism is injective. By Lemma~\ref{lem:braided-epi-reduction}, it suffices to prove the injectivity of the induced endomorphism $\beta$ of $P$. The kernel $P$ is residually nilpotent by Lemma~\ref{lem:p-residual}. Hence, we can apply the graded injectivity criterion as Theorem~\ref{thm:graded-detection}.

\subsection{The linking-number module}

Based on graded injectivity criterion, we first focus on the $1^{st}$ graded piece $gr_1(P)$ of the graded Lie ring as well as the induced map $gr_1(\beta)$. By definition, $gr_1(P)=\gamma_1(P)/\gamma_{2}(P)=P/[P,P]$. We say that this Abelian group coincides with the $\mathbb{Z}$-valued compact-supported functions on a locally compact, Hausdorff, metrizable, zero-dimensional,
and with no isolated points space $Y$. Moreover, we say that $gr_1(P)$ also has a $\mathbb{Z}V$-module structure. The $\mathbb{Z}V$-module structure will help us get the injectivity of $gr(\beta)$.

\begin{defn}\label{def:pair-space}
Let
\[
 \Delta=\{(x,x):x\in\Cantor\},\qquad
 F_2(\Cantor)=(\Cantor\times\Cantor)\setminus\Delta.
\]
Thus $F_2(\Cantor)$ is the space of ordered pairs of distinct Cantor points.
The two-element symmetric group $\mathfrak S_2$ acts by exchanging the
coordinates. Give
\[
 Y=F_2(\Cantor)/\mathfrak S_2
\]
the quotient topology, and write
\[
 q:F_2(\Cantor)\longrightarrow Y,\qquad q(x,y)=\{x,y\},
\]
for the quotient map. Points of $Y$ are therefore unordered pairs of
distinct points; in particular, $\{x,x\}$ is not a point of $Y$.

Give $\mathbb Z$ the discrete topology, and put
\[
 M=C_c(Y,\mathbb Z).
\]
Here $C_c(Y,\mathbb Z)$ denotes the additive group of continuous functions
$f:Y\to\mathbb Z$ whose support
\[
 \operatorname{supp}(f)
   =\overline{\{\eta\in Y:f(\eta)\ne0\}}^{\,Y}
\]
is compact. Continuity into a discrete space is equivalent to local
constancy: each $\eta\in Y$ has an open neighborhood on which $f$ is
constant. Moreover, $f^{-1}(\mathbb Z\setminus\{0\})$ is both open and
closed, so
\[
 \operatorname{supp}(f)=\{\eta\in Y:f(\eta)\ne0\}.
\]
Addition and negation in $M$ are pointwise:
$(f+g)(\eta)=f(\eta)+g(\eta)$ and $(-f)(\eta)=-f(\eta)$.
\end{defn}

For disjoint clopen subsets $A,B\subseteq\Cantor$, define
\[
 Y_{A,B}:=q(A\times B)
  =\{\{x,y\}\in Y:x\in A,\ y\in B\},\qquad
 e_{A,B}:=\mathbf 1_{Y_{A,B}}.
\]
The notation $\mathbf 1_E$ means the function with value $1$ on a set $E$
and value $0$ outside it. Since the pairs are unordered,
$Y_{A,B}=Y_{B,A}$ and $e_{A,B}=e_{B,A}$. If one of $A,B$ is empty, then
$e_{A,B}=0$. The restriction $q|_{A\times B}$ is a homeomorphism onto
$Y_{A,B}$. In particular, $Y_{A,B}$ is compact. It is open because
\[
 q^{-1}(Y_{A,B})=(A\times B)\sqcup(B\times A)
\]
is open, and it is closed by the same calculation with complements.
Consequently $e_{A,B}\in M$.

By the definition of the topology on Y, we have the following fact:

\begin{lem}\label{lem:pair-space-topology}
The space $Y$ is locally compact, Hausdorff, metrizable and zero-dimensional,
and it has no isolated points. It is not compact, so it is not a Cantor
space. For every clopen $A\subseteq\Cantor$, the same-side pair space
\[
 Y_A=q\bigl((A\times A)\setminus\Delta\bigr)
     =\{\{x,y\}\in Y:x,y\in A\}
\]
is clopen in $Y$. If $A\ne\varnothing$, then $Y_A$ is noncompact.
\end{lem}

\begin{rem}
The reason why we consider space $Y$ and the Abelian group $M$ is related to the expression of elements in $P$. one element in $P$ can be represented by a tree-pure braid-tree diagram. The initial and terminal of each strand are both leaf of a tree which corresponds to a cylinder clopen set of $\mathfrak{C}$. When passing to the Abelianization, the linking-number of every two strands coincides with characteristic function on the clopen set of $Y$.
\end{rem}

For every clopen $A\subseteq\Cantor$, let $A^c=\Cantor\setminus A$ and
define its \emph{cut function} by
\[
 c_A=e_{A,A^c}.
\]
Thus $c_A(\{x,y\})=1$ exactly when one of $x,y$ lies in $A$ and the
other lies outside $A$. In particular $c_A=c_{A^c}$ and
$c_\varnothing=c_\Cantor=0$. These functions belong to $M$, since their
supports are crossing-pair spaces $Y_{A,A^c}$, rather than same-side spaces.

\begin{lem}\label{lem:pair-functions}
For every finite family $f_1,\ldots,f_r\in M$, there is a finite cylinder
partition $A_1,\ldots,A_n$ of $\Cantor$ such that each $f_s$ is constant
on $Y_{A_i,A_j}$ for $i\ne j$ and is zero on each $Y_{A_i}$.
Consequently there are integers $a_{s,ij}$ with
\begin{equation}\label{eq:finite-pair-function}
 f_s=\sum_{1\le i<j\le n}a_{s,ij}e_{A_i,A_j}.
\end{equation}
In particular, the functions $e_{A,B}$ for disjoint clopen $A,B$ generate
$M$ as an abelian group, and
\[
 \operatorname{supp}(f_s)
    =\bigcup_{\substack{i<j\\a_{s,ij}\ne0}}Y_{A_i,A_j}.
\]
\end{lem}
\begin{proof}
Fix $s$ and let $K_s=\operatorname{supp}(f_s)$, which is compact in $Y$. Let $\widehat Y=(\mathfrak{C}\times\mathfrak{C})/\mathfrak{S_2}$. 
The set $\widetilde K_s=q^{-1}(K_s)$ is a compact subset of
$F_2(\Cantor)$ and is disjoint from $\Delta$. To see compactness directly,
regard $K_s$ as a compact, hence closed, subset of $\widehat Y$ and take
its inverse image under the quotient map from the compact space
$\Cantor\times\Cantor$ to $\widehat Y$.

The function $f_s\circ q$ on $F_2(\Cantor)$ is symmetric and locally
constant. Extend it to a function $\widetilde f_s$ on
$\Cantor\times\Cantor$ by assigning the value $0$ on $\Delta$.
This extension remains locally constant: the complement of the compact
set $\widetilde K_s$ is an open neighborhood of $\Delta$ on which it
vanishes. Away from $\Delta$, local constancy follows from that of $f_s$.

The compact space $\Cantor\times\Cantor$ has a finite cover by products
of cylinders on which $\widetilde f_s$ is constant. Refine their cylinder
lengths to one common length. After increasing that length for all
$s=1,\ldots,r$, every $\widetilde f_s$ is constant on every product of
two cylinders of that length. If the two cylinders agree, their product
meets the diagonal, where $\widetilde f_s$ is zero; hence its constant
value is zero on the entire product. Taking all cylinders of this common
length as the $A_i$ proves the assertion. Symmetry gives the indexing
$i<j$ in \eqref{eq:finite-pair-function}. The support formula follows
because the crossing-pair blocks are disjoint.
\end{proof}

\begin{ex}
For the partition $A_1=[0]$, $A_2=[10]$, $A_3=[11]$, the function
\[
 f=2e_{A_1,A_2}-3e_{A_2,A_3}
\]
has value $2$ on pairs crossing from $[0]$ to $[10]$, value $-3$ on
pairs crossing from $[10]$ to $[11]$, and value $0$ on every other pair.
Its support is the compact clopen set
$Y_{A_1,A_2}\sqcup Y_{A_2,A_3}$. It need not have finite support as a
set of points: compact support here can contain infinitely many pairs.
\end{ex}

We first prove that as an Abelian group $P_{ab}$ is isomorphic to $M$.

\begin{prop}\label{prop:linking-module}
There is a surjective homomorphism $\lambda:P\to M$ with kernel
$[P,P]$. It induces an isomorphism of abelian groups
\[
 \Lambda:P_{\mathrm{ab}}=P/[P,P]\xrightarrow{\cong}M,
 \qquad \Lambda\bigl(p[P,P]\bigr)=\lambda(p).
\]
More explicitly, if $p=[T,p_T,T]\in P$, where $T$ has leaf cylinders
$A_1,\ldots,A_n$ and $p_T\in\PB_n$, then
\begin{equation}\label{eq:linking-function}
 \lambda(p)=\sum_{1\le i<j\le n}\ell_{ij}(p_T)e_{A_i,A_j}.
\end{equation}
Thus its value on a pair in $Y_{A_i,A_j}$ is the linking number of
strands $i,j$, and its value on every $Y_{A_i}$ is zero.

\end{prop}
\begin{proof}
From \eqref{eq:pb-abelianization}: the classes of
the standard generators $A_{ij}$ form a basis for
$(\PB_n)_{\mathrm{ab}}\cong\mathbb Z^{\binom n2}$, and the linking
numbers are its coordinates. See \cite[Section~2]{CohenPrassidis2006}.
We check explicitly how this calculation passes through the strand-cloning
describing $P$.

Suppose that the leaf cylinder $A_i$ is split as
$A_i=A_{i0}\sqcup A_{i1}$, where $A_{i0},A_{i1}$ are its two child
cylinders. Cloning strand $i$ gives two parallel descendants with zero
mutual linking. The linking number of either descendant with a strand
$j\ne i$ is the old linking number $\ell_{ij}(p_T)$; linking numbers
between two unchanged strands are unchanged. In particular, on the
abelianization the basis element for the pair $i,j$ is sent to the sum
of the two basis elements for the pairs $i0,j$ and $i1,j$. This is
exactly the identity
\[
 e_{A_i,A_j}=e_{A_{i0},A_j}+e_{A_{i1},A_j}.
\]
The newly created block $Y_{A_{i0},A_{i1}}$ has value zero, agreeing with
the value before refinement on pairs in $Y_{A_i}$. These statements
cover every pair of new leaves, so \eqref{eq:linking-function} is
unchanged by expansion. Common refinements show that it is independent
of the representative of $p$.

To check the homomorphism property, represent $p,p'\in P$ at a common
tree $T$. Pairwise linking numbers satisfy
$\ell_{ij}(p_Tp'_T)=\ell_{ij}(p_T)+\ell_{ij}(p'_T)$. Consequently
\[
 \lambda(pp')=\lambda(p)+\lambda(p').
\]
For surjectivity, take $f\in M$ and use
Lemma~\ref{lem:pair-functions} to write it as
$\sum_{i<j}a_{ij}e_{A_i,A_j}$ for the leaf partition of one tree $T$.
Choose a pure braid having linking coordinates $a_{ij}$, for example a
product of the powers $A_{ij}^{a_{ij}}$ in any fixed order. Formula
\eqref{eq:linking-function} sends that element of $P$ to $f$.

Because $M$ is abelian, the homomorphism $\lambda$ kills $[P,P]$.
Conversely, if $\lambda([T,p_T,T])=0$, evaluate this zero function on a
pair with one point in $A_i$ and one in $A_j$, for every $i<j$.
All linking numbers $\ell_{ij}(p_T)$ vanish. The finite-stage
abelianization therefore gives $p_T\in[\PB_n,\PB_n]$, and the inclusion
$\PB_T\hookrightarrow P$ sends this commutator subgroup into $[P,P]$.
Thus $\ker\lambda=[P,P]$, and the first isomorphism theorem gives
$\Lambda$.

\end{proof}

\begin{defn}\label{def:pair-pushforward}
If $h\in\Homeo(\Cantor)$, its induced homeomorphism of $Y$ is
\[
 h^{(2)}:Y\longrightarrow Y,\qquad
 h^{(2)}(\{x,y\})=\{h(x),h(y)\}.
\]
Its inverse is $(h^{-1})^{(2)}$. Define the additive pushforward map
$h_*:M\to M$ by
\begin{equation}\label{eq:spatial-pushforward}
 (h_*f)(\{x,y\})=f\bigl(\{h^{-1}(x),h^{-1}(y)\}\bigr).
\end{equation}
\end{defn}

This formula defines an automorphism of the additive group $M$. Indeed,
composition with a homeomorphism preserves local constancy, and
\[
 \operatorname{supp}(h_*f)
       =h^{(2)}\bigl(\operatorname{supp}(f)\bigr)
\]
is compact. The map preserves addition pointwise, and
$(h^{-1})_*$ is its inverse. For homeomorphisms $h,g$, direct substitution
in \eqref{eq:spatial-pushforward} gives
\begin{equation}\label{eq:pair-pushforward-rules}
 (hg)_*=h_*g_*,\qquad
 h_*e_{A,B}=e_{h(A),h(B)},\qquad h_*c_A=c_{h(A)}.
\end{equation}
Restricting this action to $V\le\Homeo(\Cantor)$ makes $M$ a
$\mathbb ZV$-module. A $\mathbb ZV$-module endomorphism is
therefore an additive map $L:M\to M$ satisfying $Lv_*=v_*L$ for every
$v\in V$.

Moreover, the conjugation of $V$ on $P_{ab}$ induces a $\mathbb ZV$-module structure on $P_{ab}$. And $\Lambda$ is a $\mathbb{Z}V$-module isomorphism.

\begin{lem}
\label{lem:linking-conjugation}
Let $b\in BV$, put $v=\pi(b)\in V$, and let $p\in P$.
Then, for all distinct points $x,y\in\Cantor$,
\begin{equation}\label{eq:linking-equivariance}
 \lambda(bpb^{-1})(\{v(x),v(y)\})=\lambda(p)(\{x,y\}),
 \qquad x\ne y.
\end{equation}
Equivalently, $\lambda(bpb^{-1})=v_*\lambda(p)$.
Conjugation induces a well-defined action of $V$ on $P_{\mathrm{ab}}$,
and $\Lambda$ is a $\mathbb{Z}V$-module isomorphism.
\end{lem}
\begin{proof}
We first compute at one finite braid stage. Let $a\in B_n$, let
$\rho\in\mathfrak S_n$ be its endpoint permutation, and let
$q\in\PB_n$. For distinct input labels $i,j$, write $C_{ij}(a)$ for
the algebraic number of crossings between the strands starting at
$i$ and $j$: each positive crossing contributes $1$ and each negative
crossing contributes $-1$. These coordinates are indexed by unordered
pairs, so $C_{ij}=C_{ji}$. For a pure braid,
$C_{ij}(q)=2\ell_{ij}(q)$.

In the product $aqa^{-1}$, trace the strands starting at
$\rho(i),\rho(j)$ from right to left, in our composition convention.
They pass through $a^{-1}$ to positions $i,j$, remain at those
positions after the pure braid $q$, and return through $a$ to
$\rho(i),\rho(j)$. Their crossings therefore satisfy
\begin{align}
 C_{\rho(i),\rho(j)}(aqa^{-1})
  &=C_{\rho(i),\rho(j)}(a^{-1})+C_{ij}(q)+C_{ij}(a),
       \label{eq:crossing-conjugation}\\
 C_{\rho(i),\rho(j)}(a^{-1})&=-C_{ij}(a).
       \label{eq:crossing-inverse}
\end{align}
The second equality holds because reversing the braid reverses every
crossing of the corresponding two strands. The outer contributions
cancel. Since $aqa^{-1}$ is pure, dividing the remaining equality by
$2$ gives
\begin{equation}\label{eq:finite-linking-conjugation}
 \ell_{\rho(i),\rho(j)}(aqa^{-1})=\ell_{ij}(q).
\end{equation}
Thus conjugation permutes linking coordinates by the endpoint
permutation.

For the tree diagrams, pass to common expansions so that the pure
braid for $p$ is on the source leaf partition
$A_1,\ldots,A_n$ of $v$, and write $D_1,\ldots,D_n$ for the target
leaf partition. If $a$ is the braid for $b$, then
$v(A_i)=D_{\rho(i)}$, and the pure braid in the conjugate diagram is
$aqa^{-1}$. Equation~\eqref{eq:finite-linking-conjugation} and
the definition of $\lambda$ give
\begin{align*}
 \lambda(bpb^{-1})
  &=\sum_{i<j}\ell_{ij}(q)
                       e_{D_{\rho(i)},D_{\rho(j)}}\\
  &=\sum_{i<j}\ell_{ij}(q)e_{v(A_i),v(A_j)}
    =v_*\lambda(p).
\end{align*}
The symmetry $e_{A,B}=e_{B,A}$ removes any ordering issue when
$\rho(i)>\rho(j)$. For points in one source leaf cylinder the
linking function is zero; both of their images lie in one target
leaf cylinder, where the conjugate linking function is also zero.
If the diagrams require further expansion, each old linking
coordinate is copied to its descendant pairs and the mutual linking
of two parallel siblings is zero. Thus the computation is unchanged
by refinement and applies to every pair $x\ne y$.

Finally, conjugation by $b$ gives an automorphism of $P_{\mathrm{ab}}$.
If $r\in P$, then
\[
 (rpr^{-1})p^{-1}=[r,p]\in[P,P],
\]
so conjugation by $r$ is trivial on this abelianization. Two lifts of
$v$ differ by an element of $P$, and hence give the same map. We can
therefore define
\[
 v\cdot\bigl(p[P,P]\bigr)=(bpb^{-1})[P,P]
\]
independently of the chosen lift. The equality just computed states
$\Lambda(v\cdot[p])=v_*\Lambda([p])$, where $[p]=p[P,P]$.
Together with Proposition~\ref{prop:linking-module}, this proves the
$\mathbb ZV$-module isomorphism.
\end{proof}

\begin{ex}
Choose one tree with leaf cylinders $A_1,A_2,A_3$ and let
$p$ be represented there by $A_{13}^{\,2}A_{23}^{\,-3}$.
Its linking coordinates are
$(\ell_{12},\ell_{13},\ell_{23})=(0,2,-3)$, so
\[
 \lambda(p)=2e_{A_1,A_3}-3e_{A_2,A_3}.
\]
Let $b$ be represented on that tree by $\sigma_1$, so its image $v$
interchanges $A_1,A_2$ and fixes $A_3$. The lemma gives
\[
 \lambda(bpb^{-1})=2e_{A_2,A_3}-3e_{A_1,A_3},
\]
and the new coordinate vector is $(0,-3,2)$. The crossings of
$\sigma_1$ and $\sigma_1^{-1}$ cancel in the conjugate; they do not
add a linking number to the pure braid.
\end{ex}

Now, we can identify $P_{ab}$ with $M$ as a $\mathbb ZV$-module. The benefit of this is that the ring of endomorphisms of $M$ has a very clear description as follows:

\begin{lem}\label{lem:scalar-commutant}
Every $V$-equivariant endomorphism of the additive group $M$ is
multiplication by a single integer. Equivalently, the map
\[
 \mathbb Z\longrightarrow\operatorname{End}_{\mathbb ZV}(M),
 \qquad a\longmapsto(f\mapsto af)
\]
is a ring isomorphism, where addition of endomorphisms is pointwise and
their multiplication is composition.
\end{lem}
\begin{proof}
Let $L:M\to M$ be additive and satisfy $Lv_*=v_*L$ for every $v\in V$.
Fix a nonempty proper clopen set $A\subset\Cantor$, and consider
\[
 H_A=V_A V_{A^c}\cong V_A\times V_{A^c}.
\]
The factors commute and have trivial intersection because their supports
are disjoint. The cut function $c_A$ is invariant under $H_A$.
Equivariance therefore makes $L(c_A)$ invariant under $H_A$ as well.

There is a partition into three clopen regions
\[
 Y=Y_A\sqcup Y_{A,A^c}\sqcup Y_{A^c}.
\]
Every $H_A$-orbit is dense in the region containing it. On $Y_A$ this
is the unordered-pair assertion for $V_A$ in
Lemma~\ref{lem:local-density}; on $Y_{A^c}$ it is the corresponding
assertion for $V_{A^c}$. On $Y_{A,A^c}$, identify a pair with its unique
ordered representative in $A\times A^c$. The two factors act
independently on the two coordinates. A basic open set is a product
of a nonempty open subset of $A$ and a nonempty open subset of $A^c$,
and the one-point density assertion in Lemma~\ref{lem:local-density}
moves each coordinate into its prescribed open set.

An invariant continuous integer-valued function is constant on each
such region. Indeed, fix a point in the region. The fiber of its
function value is closed and contains its dense orbit, so it contains
the entire region. Thus $L(c_A)$ takes a constant value on each of
$Y_A$, $Y_{A,A^c}$ and $Y_{A^c}$.

By Lemma~\ref{lem:pair-space-topology}, $Y_A$ and $Y_{A^c}$ are closed
noncompact subsets of $Y$. If the constant value of $L(c_A)$ on either
one were nonzero, its compact support would contain that whole region.
The region would then be a closed subset of a compact space, and hence
compact, a contradiction. The only possibly nonzero value is on
$Y_{A,A^c}$. There is consequently an integer $a_A$ such that
\begin{equation}\label{eq:scalar-on-cuts}
 L(c_A)=a_Ac_A.
\end{equation}

Let $B$ be any other nonempty proper clopen set. By
Lemma~\ref{lem:partition-matching}, choose $v\in V$ with $v(A)=B$.
Equation~\eqref{eq:pair-pushforward-rules} and equivariance give
\[
 a_Bc_B=L(c_B)=L(v_*c_A)=v_*L(c_A)=a_Ac_B.
\]
Since $c_B$ takes the value $1$, we obtain $a_A=a_B$. Denote this
common integer by $a$. The equality $L(c_A)=ac_A$ also holds for
$A=\varnothing,\Cantor$, because those cut functions are zero.

For disjoint clopen sets $A,B$ the following pointwise identity holds:
\begin{equation}\label{eq:cut-identity}
 2e_{A,B}=c_A+c_B-c_{A\cup B}.
\end{equation}
For completeness, put $C=\Cantor\setminus(A\cup B)$. On a pair
crossing $A$ and $B$, the right side is $1+1-0=2$. On a pair crossing
$A$ and $C$, or $B$ and $C$, it is $1+0-1=0$, or $0+1-1=0$,
respectively. On pairs contained in any one of $A,B,C$, all three
cut functions vanish. These possibilities exhaust all pairs.
The same calculation includes $C=\varnothing$.

Apply $L$ to \eqref{eq:cut-identity} and use its value on cut functions:
\[
 2L(e_{A,B})=a(c_A+c_B-c_{A\cup B})=2ae_{A,B}.
\]
The additive group $M$ is torsion-free: if $nf=0$ for a nonzero
integer $n$, then $nf(\eta)=0$ in $\mathbb Z$ for every $\eta\in Y$,
and hence $f=0$. We may therefore cancel the factor $2$ and conclude
that $L(e_{A,B})=ae_{A,B}$. These indicators generate $M$ by
Lemma~\ref{lem:pair-functions}, so $L=a\,\mathrm{id}_M$.
Conversely, every integer scalar commutes with the $V$-action.
Distinct scalars give distinct maps since a nonzero rectangle indicator
takes the value $1$. This also proves the stated ring isomorphism.
\end{proof}

Let $\Phi: BV\twoheadrightarrow BV$ be a surjective endomorphism, let
$\beta=\Phi|_P:P\twoheadrightarrow P$, and let $\alpha\in\Aut(V)$
be the quotient automorphism supplied by Lemma~\ref{lem:braided-epi-reduction}.
There is $h\in\Homeo(\Cantor)$ satisfying
$\alpha(v)=hvh^{-1}$ for every $v\in V$.

\begin{prop}\label{prop:signed-spatial}
 For any such $h$, there is
$\varepsilon\in\{1,-1\}$ such that
\begin{equation}\label{eq:signed-spatial-action}
 \Lambda\circ\beta_{\mathrm{ab}}\circ\Lambda^{-1}
       =\varepsilon h_*.
\end{equation}
Here $\beta_{\mathrm{ab}}:P/[P,P]\to P/[P,P]$ is the induced map
$p[P,P]\mapsto\beta(p)[P,P]$. In the identification
$P_{\mathrm{ab}}=M$, formula \eqref{eq:signed-spatial-action} is written
$\beta_{\mathrm{ab}}=\varepsilon h_*$. Equivalently, for every $p\in P$
and every $x\ne y$ in $\Cantor$,
\begin{equation}\label{eq:signed-linking-pointwise}
 \lambda(\beta(p))(\{x,y\})
  =\varepsilon\lambda(p)
           (\{h^{-1}(x),h^{-1}(y)\}).
\end{equation}
\end{prop}
\begin{proof}

Transport the abelianized map to $M$ by setting
\[
 \overline\beta=\Lambda\circ\beta_{\mathrm{ab}}\circ\Lambda^{-1}.
\]
This additive map is surjective because $\beta$ is surjective. We
compare its equivariance with that of $h_*$. For $v\in V$, choose a
lift $b\in BV$ with $\pi(b)=v$. For $p\in P$, the homomorphism
property of $\Phi$ gives
\[
 \beta(bpb^{-1})=\Phi(b)\beta(p)\Phi(b)^{-1}.
\]
Moreover, the definition of the quotient map gives
$\pi(\Phi(b))=\alpha(v)$. Apply $\lambda$ and use
Lemma~\ref{lem:linking-conjugation} on both sides. Since $\lambda$ is
surjective, the resulting identity is
\begin{equation}\label{eq:beta-twisted-equivariance}
 \overline\beta\,v_* =\alpha(v)_*\,\overline\beta
 \qquad(v\in V).
\end{equation}

Independently, the composition rule for pushforwards and
$\alpha(v)=hvh^{-1}$ give
\[
 \alpha(v)_*=h_*v_*h_*^{-1},\qquad
 h_*v_* =\alpha(v)_*h_*.
\]
Define the additive endomorphism
$L=h_*^{-1}\overline\beta$ of $M$. Combining the preceding identities
with \eqref{eq:beta-twisted-equivariance}, we obtain
\[
 Lv_*=h_*^{-1}\overline\beta\,v_*
     =h_*^{-1}\alpha(v)_*\overline\beta
     =v_*h_*^{-1}\overline\beta=v_*L.
\]
Thus $L$ is a $\mathbb ZV$-module endomorphism. The scalar commutant
calculation, Lemma~\ref{lem:scalar-commutant}, gives
$L=a\,\mathrm{id}_M$ for some $a\in\mathbb Z$. Therefore
\[
 \overline\beta=ah_*.
\]

Since both $\overline\beta$ and $h_*^{-1}$ are surjective, so is
$L=a\,\mathrm{id}_M$. Choose nonempty disjoint clopen sets $A,B$.
Surjectivity supplies $f\in M$ with $af=e_{A,B}$. At any point
$\eta\in Y_{A,B}$, evaluation gives $af(\eta)=1$. Both factors are
integers, so $a=1$ or $a=-1$. Set $\varepsilon=a$; this proves
\eqref{eq:signed-spatial-action}. Applying that equality to
$\Lambda(p[P,P])=\lambda(p)$ yields
\eqref{eq:signed-linking-pointwise}.
\end{proof}

\subsection{Application of the graded injectivity criterion}
\label{subsec:bv-graded-application}

We now apply Section~\ref{sec:graded-preliminaries} to the pure braid
kernel $P$. Lemma~\ref{lem:p-residual} proves that $P$ is residually
nilpotent. To use Theorem~\ref{thm:graded-detection}, it remains to
prove injectivity in every graded degree. The following
finite-realization lemma lets us control each graded piece by an
action of $V$ agreeing with the spatial homeomorphism on the finitely
many degree-one entries involved.

\begin{lem}\label{lem:finite-realization}
Let $h\in\Homeo(\Cantor)$ and $f_1,\ldots,f_r\in M$. There exists $v\in V$
such that $v_*f_i=h_*f_i$ for every $i$.
\end{lem}
\begin{proof}
Choose a partition $A_1,\ldots,A_n$ as in
Lemma~\ref{lem:pair-functions} for all $f_i$. By
Lemma~\ref{lem:partition-matching}, choose $v\in V$ with
$v(A_j)=h(A_j)$ for every $j$. Each $f_i$ is a linear combination of
$e_{A_j,A_k}$, and both pushforwards send this indicator to
$e_{h(A_j),h(A_k)}$. Hence the pushforwards agree on all $f_i$.
\end{proof}

\begin{thm}\label{thm:bv-hopfian}
The braided Thompson group $ BV$ is Hopfian.
\end{thm}
\begin{proof}
Let $\Phi: BV\twoheadrightarrow BV$. By
Lemma~\ref{lem:braided-epi-reduction}, it is enough to prove injectivity of
$\beta=\Phi|_P$. Proposition~\ref{prop:signed-spatial} gives
\[
 \gr_1(\beta)=\beta_{\mathrm{ab}}=\varepsilon h_*,\qquad\varepsilon=\pm1.
\]
Fix $k\ge1$ and a nonzero $z\in\gr_k(P)$. By
Lemma~\ref{lem:degree-one}, write
\[
 z=F(m_1,\ldots,m_r),\qquad m_i\in M,
\]
where $F$ is a finite integral sum of iterated Lie brackets, all of length
$k$. Choose $v\in V$ with $v_*m_i=h_*m_i$ for every $i$, using
Lemma~\ref{lem:finite-realization}. Lift $v$ to $b\in BV$.
Since $P$ is normal, ambient conjugation restricts to the automorphism
$c_b:P\to P$, $c_b(p)=bpb^{-1}$.
Proposition~\ref{prop:induced-graded-map} gives the graded Lie
automorphism $C_b=\gr(c_b)$ of $\gr(P)$; its inverse is
$\gr(c_{b^{-1}})$. By Lemma~\ref{lem:linking-conjugation},
its degree-one action is $v_*$. 
Functoriality of the Lie bracket and the degree-$k$ homogeneity
formula \eqref{eq:homogeneous-scaling} give
\begin{align*}
 \gr_k(\beta)(z)
 &=F(\varepsilon h_*m_1,\ldots,\varepsilon h_*m_r)\\
 &=\varepsilon^kF(v_*m_1,\ldots,v_*m_r)
  =\varepsilon^k C_b(z)\neq0.
\end{align*}
The last expression is nonzero because $C_b$ is injective and
multiplication by $\varepsilon^k\in\{1,-1\}$ is an automorphism of
the abelian group $\gr_k(P)$. The element $v$ may depend on $z$; this causes no problem, since this
calculation proves injectivity on each individual nonzero graded class.
Thus every $\gr_k(\beta)$ is injective. By
Lemma~\ref{lem:p-residual} and Theorem~\ref{thm:graded-detection},
$\beta$ is injective. Therefore $\Phi$ is injective as well.
\end{proof}

\begin{rem}
The proof requires neither a splitting of \eqref{eq:bv-extension} nor an
automorphism of $ BV$ lifting $h$. It also does not require the preservation
of the kernels of all individual strand-deletion maps. The finite action of
$h$ on a collection of linking-number functions is sufficient.
\end{rem}

\section{Hopficity of the braided Thompson group \texorpdfstring{$BF$}{BF}}
\label{sec:bf-hopficity}

The goal of this section is to prove that $BF$ is Hopfian. We use the same strategy for $BV$. The main difference is that in this case $M$ is a $\mathbb Z F$-module, so the ring of endomorphisms of $M$ has different structure.

\subsection{Finite ordered partitions of the linking-number module}

Retain the space $Y$ of unordered pairs of distinct Cantor points
and the abelian group $M=C_c(Y,\mathbb Z)$ from
Section~\ref{sec:bv-hopficity}. Under Proposition~\ref{prop:linking-module},
we identify $P_{\mathrm{ab}}$ with $M$. Restricting the conjugation
action from $BV$ to $BF$ makes it a $\mathbb ZF$-module. Explicitly,
$f_*m(\{x,y\})=m(\{f^{-1}x,f^{-1}y\})$. A
$\mathbb ZF$-endomorphism means an additive map commuting with
every such $f_*$; write $\operatorname{End}_{\mathbb ZF}(M)$ for
the ring of these maps, with composition as multiplication.

Let us recall the existence of a dense orbit of the local group of $F$. In particular, we have the following Lemma, one can find the proof in \cite{CannonFloydParry1996}.

\begin{lem}\label{lem:bf-local-density}
Suppose $U$ is a cylinder. Ordered partitions of $U$ into the same
number of nonempty convex clopen subsets can be matched by an
element of $F$ supported in $U$.
The local subgroup
\[
 F_U=\{f\in F:f|_{\Cantor\setminus U}=\operatorname{id}\}
\]
has a dense orbit on $U$ and a dense orbit on the unordered pairs
of distinct points of $U$.
\end{lem}

For a tree $T$ with ordered leaf cylinders $A_1<\cdots<A_n$, let
\begin{equation}\label{eq:bf-finite-module}
 M_T=\bigoplus_{1\le i<j\le n}\mathbb Z e_{A_i,A_j}\subseteq M.
\end{equation}
These are finite-rank free abelian groups, not $F$-submodules.
Every element of $M$ belongs to some $M_T$ by
Lemma~\ref{lem:pair-functions}. Splitting a leaf embeds $M_T$ into
the group for the refined tree by the linking-number cloning rule. We first focus on the endomorphism restricting on such subgroups.

\begin{lem}
\label{lem:bf-stage-preservation}
For every $\tau\in\operatorname{End}_{\mathbb ZF}(M)$ and every
tree $T$, one has $\tau(M_T)\subseteq M_T$.
Under the ordered basis in \eqref{eq:bf-finite-module}, the matrix
of $\tau|_{M_T}$ depends only on the number $n$ of leaves of $T$.
\end{lem}
\begin{proof}
Let $H_T=F_{A_1}\cdots F_{A_n}$, the product of the commuting
local subgroups on the leaves of $T$. Each $e_{A_i,A_j}$ is fixed
by $H_T$. Therefore its image under $\tau$ is fixed by $H_T$. We consider the function in $M$ which are fixed by $H_T$. Let $m \in M$ such that $h_*m=m$ for all $h \in H_T$. 

On a crossing block $Y_{A_r,A_s}$, the product
$F_{A_r}\times F_{A_s}$ has a dense orbit, by the existence of a dense orbit of local group in Lemma~\ref{lem:bf-local-density}. Then $m$ as a $H_T$-invariant function is constant on that block: a level set containing
the dense orbit is closed and contains the whole block.

On the same-side block $Y_{A_r}$, the the existence of a dense orbit shows that an $H_T$-invariant function is constant. This constant must be zero, since $Y_{A_r}$ is noncompact by
Lemma~\ref{lem:pair-space-topology}, whereas an element of $M$ has compact support. More explicitly, $Y_{A_r}$ is closed in $Y$; if a nonzero constant occurred there, it would be a closed subset of a compact support and hence compact, a contradiction. 

These crossing and same-side blocks partition $Y$, proving that the $m$ belongs to $M_T$. Hence the image of $\tau$ belongs to $M_T$.

For any two $n$-leaf trees, their leaf partitions are matched in order by an element of $F$. Its action carries the indexed basis of one $M_T$ to that of the other. Commutation with this action gives the same matrix for $\tau$ in both ordered bases.
\end{proof}

\subsection{The four endomorphisms of the linking-number module}

The Lemma~\ref{lem:scalar-commutant} for $V$ cannot simply be reused for $F$.
We compute the larger endomorphism ring explicitly.

By Lemma~\ref{lem:bf-stage-preservation}, we can forget the tree $T$. For $n\ge2$, write
\[
 E_n=\bigoplus_{1\le i<j\le n}\mathbb Z e_{ij}
\]
for the abstract ordered array model of $M_T$, and write $f_{ij}$
for the coordinate of $f\in E_n$. Set $f_{ji}=f_{ij}$ and
$f_{ii}=0$. Write $I$ for the identity map of $M$, and let $I_n$
be the identity map of $E_n$. Define three additive
maps $L_n,R_n,J_n:E_n\to E_n$ by
\begin{align}
 (L_nf)_{ij}&=f_{1j}-f_{1i},\label{eq:bf-left-operator}\\
 (R_nf)_{ij}&=f_{in}-f_{jn},\label{eq:bf-right-operator}\\
 (J_nf)_{ij}&=\sum_{r=i}^{j-1}f_{r,r+1}
 \qquad(1\le i<j\le n).\label{eq:bf-adjacent-operator}
\end{align}
The subscripts $1,n$ record the first and last leaves, while the
sum in $J_n$ records consecutive pairs. At a one-leaf tree the
module is zero, and all these maps are the zero map.

Next, we extend three classes of endomorphisms of finite rank free Abelian group to endomorphisms of $M$. Let $T$ be a finite rooted binary tree with ordered leaves
$A_1<\cdots<A_n$, and let $T'$ be obtained by attaching a caret to the $k$-th leaf. Number the leaves of $T'$ as
$A_1<\cdots<A_{n+1}$. Such splitting induces the ordered cloning $\iota_{T,T'}\colon M_T\longrightarrow M_{T'}$
which is determined on the standard basis by
\[
\iota_{T,T'}(e_{A_i,A_j})=
\begin{cases}
e_{A_i,A_j},
    & i<j<k,\\[2mm]
e_{A_i,A_k}+e_{A_i,A_{k+1}},
    & i<j=k,\\[2mm]
e_{A_i,A_{j+1}},
    & i<k<j,\\[2mm]
e_{A_k,A_{j+1}}+e_{A_{k+1},A_{j+1}},
    & k=i<j,\\[2mm]
e_{A_{i+1},A_{j+1}},
    & k<i<j.
\end{cases}
\]
And $M=\varinjlim_T M_T$ is the direct limit whose transition maps are the cloning inclusions $\iota_{T,T'}$. So we show that three classes of endomorphisms can be extend to $M$ by showing that all such endomorphisms commute with ordered cloning inclusion.

\begin{lem} \label{lem:L,R,J-endomorphisms}
The maps $L_n,R_n,J_n$ commute with every ordered cloning inclusion.
They therefore define maps $L,R,J\in\operatorname{End}_{\mathbb ZF}(M)$.
Their common image is the subgroup $C\subseteq M$ whose finite-stage
arrays have the form
\begin{equation}\label{eq:bf-gradient-array}
 f_{ij}=u_j-u_i\quad(1\le i<j\le n),
 \qquad u_1,\ldots,u_n\in\mathbb Z.
\end{equation}
Each of $L,R,J$ restricts to the identity on $C$. 
\end{lem}
\begin{proof} 
For a cloning of leaf $k$ in an $n$-leaf partition, let
$p:\{1,\ldots,n+1\}\to\{1,\ldots,n\}$ be the order-preserving
surjection identifying $k,k+1$. Denote the cloning inclusion by $\iota_{n,k}$, the commutation means that $O_{n+1}\circ \iota_{n,k}=\iota_{n,k}\circ O_n$, for $O\in \{L,R,J\}$.  After applying $\iota_{n,k}$ the new arrays are denoted by $f'_{ij}$, where
\begin{equation}\label{eq:bf-array-cloning}
 f'_{ij}=f_{p(i),p(j)};
\end{equation}
the convention $f_{rr}=0$ accounts for the zero linking between the two clones. Since $p(1)=1$ and $p(n+1)=n$,
\[
 (L_{n+1}f')_{ij}=f_{1,p(j)}-f_{1,p(i)},\qquad
 (R_{n+1}f')_{ij}=f_{p(i),n}-f_{p(j),n}.
\]
These are precisely the cloned arrays of $L_nf$ and $R_nf$;
they also vanish on the new clone pair. 

In the sum for $J_{n+1}$,
each old consecutive pair between $p(i)$ and $p(j)$ occurs once,
and the additional clone pair contributes zero. Thus
\[
 (J_{n+1}f')_{ij}
   =\sum_{r=p(i)}^{p(j)-1}f_{r,r+1},
\]
which is the cloned array of $J_nf$. This proves compatibility.
The ordered action of $F$ matches equally indexed leaves after
common refinement. The formulas depend only on the indices and
commute with refinement, so they also commute with $F$.

For $L_n$ take $u_i=f_{1i}$, for $R_n$ take $u_i=-f_{in}$, and for $J_n$ take $u_i=\sum_{r=1}^{i-1}f_{r,r+1}$, with the empty sum equal to zero. So all three outputs have form \eqref{eq:bf-gradient-array}. Conversely, if $f_{ij}=u_j-u_i$, then
\[
 f_{1j}-f_{1i}=u_j-u_i,
 \quad f_{in}-f_{jn}=u_j-u_i,
 \quad \sum_{r=i}^{j-1}f_{r,r+1}=u_j-u_i.
\]
Thus each map fixes all gradient arrays, and their images coincide
at every stage and in the direct limit.
\end{proof}

Next, we say that the endomorphisms $I$, $L$, $R$, and $J$ generates $\operatorname{End}_{\mathbb ZF}(M)$.

\begin{prop}\label{prop:bf-module-commutant}
Every $\tau\in\operatorname{End}_{\mathbb ZF}(M)$ has a unique
expression
\begin{equation}\label{eq:bf-four-parameter}
 \tau=aI+bL+cR+dJ\qquad(a,b,c,d\in\mathbb Z).
\end{equation}
Thus the underlying abelian group of this endomorphism ring is
free of rank four. Its multiplication is determined by $I$ being
the identity and
\begin{equation}\label{eq:bf-operator-multiplication}
 AB=B\qquad\text{for }A,B\in\{L,R,J\}.
\end{equation}
\end{prop}
\begin{proof} 
Let $\tau\in\operatorname{End}_{\mathbb ZF}(M)$ be an arbitrary endomorphism. 
By Lemma~\ref{lem:bf-stage-preservation}, $\tau$ gives an indexed map $\tau_n:E_n\to E_n$ for each number $n$ of leaves. These maps commute with each cloning inclusion
$\iota_{n,k}:E_n\to E_{n+1}$, which is the linear map in
\eqref{eq:bf-array-cloning}. 

At $n=2$, write $\tau_2(e_{12})=\ell e_{12}$, with $\ell\in\mathbb Z$. So $\tau_2$ is uniquely determined by the integer $\ell \in \mathbb Z$.

At $n=3$, cloning the first or the second leaf of $E_2$ gives
respectively $e_{13}+e_{23}$ and $e_{12}+e_{13}$. Hence
\begin{align*}
 \tau_3(e_{13}+e_{23})&=\ell(e_{13}+e_{23}),\\
 \tau_3(e_{12}+e_{13})&=\ell(e_{12}+e_{13}).
\end{align*}
Write $\tau_3(e_{13})=u e_{12}+v e_{13}+w e_{23}$.
Subtracting determines the other two basis images:
\begin{align*}
 \tau_3(e_{12})&=(\ell-u)e_{12}+(\ell-v)e_{13}-w e_{23},\\
 \tau_3(e_{23})&=-u e_{12}+(\ell-v)e_{13}+(\ell-w)e_{23}.
\end{align*}
The four integers $\ell,u,v,w$ therefore determine $\tau_2,\tau_3$.

Now we consider a new endomorphism $\tau'=aI+bL+cR+dJ$ where $a,b,c,d\in\mathbb Z$, such that the indexed maps at $n=2,3$ are the same as $\tau$: $\tau'_2=\tau_2$ and $\tau'_3=\tau_3$.
Using \eqref{eq:bf-left-operator}--\eqref{eq:bf-adjacent-operator},
the endomorphism $\tau'$ exists if we pick
\begin{equation}\label{eq:bf-parameter-conversion}
 b=w,\qquad c=u,\qquad a=v-u-w,\qquad d=\ell-v.
\end{equation}
In particular, $\ell=a+b+c+d$.
 
We prove that $\tau=\tau'$ by proving that all other indexed maps are equal: $\tau_n=\tau'_n$ for $n\ge4$.

We first prove that the images of all cloning maps
$E_{n-1}\to E_n$ generate $E_n$ as an abelian group for $n\ge4$.
Here each $e_{ij}$ in an expression is in the indicated source
or target according to its indices. For $n=4$, explicit integral
expressions are
\begin{align*}
 e_{12}&=\iota_{3,3}(e_{12}), &
 e_{14}&=\iota_{3,2}(e_{13}), &
 e_{34}&=\iota_{3,1}(e_{23}),\\
 e_{13}&=\iota_{3,3}(e_{13})-e_{14}, &
 e_{24}&=\iota_{3,1}(e_{13})-e_{14}, &
 e_{23}&=\iota_{3,1}(e_{12})-e_{13}.
\end{align*}
For $n=5$, every pair except $(2,4)$ has two consecutive indices
in its complement. Merge those consecutive indices: the resulting
cloning map sends one basis element directly to the desired
$e_{ij}$. For the exceptional pair,
\[
 e_{24}=\iota_{4,1}(e_{13})-e_{14},
\]
and $e_{14}$ is already covered by the preceding construction.
For $n\ge6$, the complement of any two indices contains two
consecutive indices. Indeed, a subset with no consecutive indices
has at most $\lceil n/2\rceil$ members, while
$n-2>\lceil n/2\rceil$. Merging two consecutive complementary
indices again makes $e_{ij}$ a direct cloning image. This proves
the generation assertion for every $n\ge4$.

So $E_n$ is generated by the image of all cloning maps for $n\ge4$. It follows inductively that maps commuting with cloning are uniquely determined by their maps at $n=2,3$. Hence, $\tau=\tau'=aI+bL+cR+dJ$. And \eqref{eq:bf-parameter-conversion} also proves uniqueness over
$\mathbb Z$. Finally, each of $L,R,J$ takes values in $C$ and
restricts to the identity there. Thus $A(Bf)=Bf$, proving
\eqref{eq:bf-operator-multiplication}.
\end{proof}

From Lemma~\ref{lem:L,R,J-endomorphisms}, we have that $J^2=J$ and the direct decomposition, with components $Jf$ and $(I-J)f$. So $M=C\oplus \ker J $. The common image at a tree with $n$ leaves is freely generated by the $n-1$ cut functions separating the first $r$ leaves from the last $n-r$ leaves, for $1\le r<n$. This helps us describe the surjective endomorphism clearly:

\begin{lem}
\label{lem:bf-module-units}
Suppose $\tau=aI+bL+cR+dJ$ is surjective on $M$, and put
$\ell=a+b+c+d$. Then
\[
 a\in\{1,-1\},\qquad \ell\in\{1,-1\}.
\]
Consequently $\tau|_{M_T}$ is an automorphism for every tree $T$,
and $\tau$ is an automorphism of $M$.
\end{lem}
\begin{proof}
Using \eqref{eq:bf-operator-multiplication}, we have that
\[
 (I-J)\tau=a(I-J).
\]
Choose the three-leaf partition $[0],[10],[11]$. Consider the function $e_{[0],[11]}$ and apply $J$ on it. We have that $Je_{[0],[11]}=J_3e_{1,3}=0$, hence $e_{[0],[11]}$ lies in $\ker J$. By surjectivity, there exists $m$ with
$\tau m=e_{[0],[11]}$. Applying $I-J$, we have that
\[
 (I-J)\tau m=a(I-J)m=(I-J)e_{[0],[11]}=e_{[0],[11]}.
\]
Since $(I-J)m$ and $e_{[0],[11]}$ are both integer-valued functions, and $e_{[0],[11]}$ only takes values in $\{0,1\}$.
This implies $a=\pm1$.

Since the image of $J$ in $M$ is $C$, for any $u\in C$, there exists $m_u\in M$ such that $Jm_u =u$. Hence by
\eqref{eq:bf-operator-multiplication}, we have that
$$\tau u=\tau Jm_u=(aI+bL+cR+dJ)Jm_u=(a+b+c+d)Jm_u=\ell u.$$

In particular, we have that $\tau|_C=\ell I|_C$. For any $c_0\in C$, using the surjectivity of $\tau$, there exists $m_0 \in M$, such that $\tau m_0=c_0$. Since $c_0\in C$, apply $I-J$ gives us $a(I-J)m_0=0$ which means that $m_0 \in C$. Therefore surjectivity of $\tau$ implies that $\tau|_C$ is surjective, $\tau C=C$ and $\ell C=C$. Now consider the function $e_{[0],[1]}$ which lies in $C$ by $Je_{[0],[1]}=e_{[0],[1]}$. Then there exists $m' \in C$ such that $\tau m'=\ell m'= e_{[0],[1]}$. As before, the evaluation implies that $\ell=\pm1$.

Put $N=bL+cR+dJ$, so $\tau=aI+N$. By
\eqref{eq:bf-operator-multiplication},
$N^2=(\ell-a)N$. Direct multiplication therefore gives
\begin{equation}\label{eq:bf-module-inverse}
 \tau^{-1}=a^{-1}I-(a\ell)^{-1}N.
\end{equation}
Both coefficients are integers because $a,\ell$ are units in
$\mathbb Z$. The same formula preserves every $M_T$, since each
operator does. Thus the restriction to every finite stage is
an automorphism.
\end{proof}

\begin{ex}
\label{ex:bf-module-extra-unit}
The map $I+L-R$ is an $F$-equivariant automorphism of $M$.
Here $a=\ell=1$ and $(L-R)^2=0$, so its inverse is $I-L+R$.
On the four-leaf array whose six coordinates all equal $1$, its
output, in the order $12,13,14,23,24,34$, is
\[
 (2,2,1,1,0,0).
\]
These coordinates are not all equal. The next two lemmas show that
this automorphism cannot be the degree-one part of a graded Lie
endomorphism preserving the pure braid stages. Thus an automorphism
of the abelianization alone does not finish the $BF$ argument.
\end{ex}

\subsection{The Lie bracket eliminates the three extra operators}

We continue to use $\gr(P)$ and its lower-central grading from
Section~\ref{sec:graded-preliminaries}. By
Lemma~\ref{lem:p-residual}, inclusion of a finite pure braid stage
induces an injective map of graded Lie rings. We therefore identify
\[
 \gr(\PB_T)\subseteq\gr(P).
\]
Its degree-one part is $M_T\cong \PB_T/[\PB_T,\PB_T] $, and it is generated as a Lie ring
by $M_T$, by Lemma~\ref{lem:degree-one}. In its abstract $n$-strand
model, denote by $t_{ij}$ the degree-one class of the standard pure
braid generator $A_{ij}$. Under the linking-number identification,
$t_{ij}$ is $e_{ij}$. The notation $t_{ij}$ emphasizes that a Lie
bracket can be applied to this class.

The center of $\gr(\PB_n)$ is determined in \cite[Corollary~1.3]{CohenPrassidis2006}

\begin{lem}(\cite[Corollary~1.3]{CohenPrassidis2006})
\label{lem:bf-stage-center}
For $n\ge3$, a degree-one element of $\gr(\PB_n)$ that commutes
with every degree-one generator is a multiple of
\begin{equation}\label{eq:bf-central-class}
 z_n=\sum_{1\le i<j\le n}t_{ij}.
\end{equation}
The element $z_n$ commutes with the entire Lie ring.
\end{lem}
\begin{proof}
By \cite[Corollary~1.3]{CohenPrassidis2006},
the center of $\gr(\PB_n)$ is $\mathbb Z z_n$.
Their $E_0^*(P_n)$ and $\Delta(n)$ correspond to our
$\gr(\PB_n)$ and $z_n$, respectively.
Since $\gr(\PB_n)$ is generated by the degree-one elements
$t_{ij}$, any element commuting with all these generators
is central.
\end{proof}

As in the proof of Theorem~\ref{thm:bv-hopfian}, every homeomorphism $h$ of $\Cantor$ can induce a graded Lie
automorphism of $\gr(P)$.

\begin{lem}
\label{lem:bf-spatial-graded}
For every homeomorphism $h$ of $\Cantor$ there is a graded Lie
automorphism $\mathcal H_h$ of $\gr(P)$ whose degree-one action
is $h_*$. It is characterized by
\begin{equation}\label{eq:bf-spatial-polynomial}
 \mathcal H_h\bigl(Q(m_1,\ldots,m_r)\bigr)
     =Q(h_*m_1,\ldots,h_*m_r),\qquad m_i\in M,
\end{equation}
for finite integral Lie polynomials $Q$. Moreover,
$\mathcal H_{h^{-1}}=\mathcal H_h^{-1}$.
\end{lem}
\begin{proof}
We use the same method as in the proof of Theorem~\ref{thm:bv-hopfian}. By Lemma~\ref{lem:degree-one}, every homogeneous class is a finite Lie polynomial in degree-one classes. We check that the proposed
substitution is independent of the polynomial expression.
For any finite collection of its inputs, choose $v\in V$ with
$v_*m_i=h_*m_i$ simultaneously, using
Lemma~\ref{lem:finite-realization}. Lift $v$ to $b\in BV$.
Conjugation by $b$ is an automorphism of $P$, hence induces a
graded Lie automorphism $C_b$; its degree-one action is $v_*$ by Lemma~\ref{lem:linking-conjugation}. On polynomials in this finite collection, substitution of $h_*m_i$ consequently agrees with
$C_b$.

If two polynomials represent the same class, apply this argument
to the combined finite collection of their inputs. Their
substituted difference is $C_b(0)=0$. The same argument, with the
inputs of two classes combined, proves additivity and preservation
of the Lie bracket. Substitution preserves degree. Doing the
substitution for $h^{-1}$ reverses it on all degree-one classes
and hence on all their Lie polynomials, giving the stated inverse.
This constructs an automorphism of the graded Lie ring. We do
not assume an automorphism of $BF$ or $BV$ lifting $h$.
\end{proof}

From Lemma~\ref{lem:bf-module-units}, we know that the surjective endomorphisms of $M$ already have rigidity in the coefficients. With property $\gr(\PB_T)\subseteq\gr(P)$, we have the stronger rigidity on the coefficients as follows.

\begin{prop}
\label{prop:bf-lie-scalar}
Let $\Gamma:\gr(P)\to\gr(P)$ be a graded Lie homomorphism whose
degree-one map $\tau$ is surjective and $F$-equivariant. Then
$\tau=\varepsilon I$ for $\varepsilon\in\{1,-1\}$, and
\[
 \Gamma|_{\gr_k(P)}=\varepsilon^k I\qquad(k\ge1).
\]
In particular, $\Gamma$ is an automorphism.
\end{prop}
\begin{proof}
By Proposition~\ref{prop:bf-module-commutant}, write
$\tau=aI+bL+cR+dJ$. Lemma~\ref{lem:bf-module-units} gives
$a=\pm1$ and makes its restriction to each $M_T$ surjective.
Lemma~\ref{lem:bf-stage-preservation} says that $\Gamma$ preserves
degree one of every finite stage: $\Gamma(M_T)\subseteq M_T$. Because $\gr(\PB_T)$ is
generated by that degree-one subgroup $M_T$, $\Gamma(\gr(\PB_T))
\subseteq\gr(\PB_T)$. Thus it restricts to a graded Lie
endomorphism $\Gamma_T$ of the finite-stage pure braid Lie ring.

Choose any four-leaf tree $T_4$. By
Lemma~\ref{lem:bf-stage-center}, the class $z_4=t_{12}+t_{13}+t_{14}+t_{23}+t_{24}+t_{34}$ is central in $\gr(PB_4)=\gr(PB_{T_4})$.
For every degree-one $m \in \gr_1(PB_4)$ there,
\[
 [\tau(z_4),\tau(m)]
    =\Gamma_{T_4}([z_4,m])=0.
\]
From the assumption and Lemma~\ref{lem:bf-module-units}, the map $\tau|_{M_{T_4}}$ is surjective for every $T_4$, so $\tau(z_4)$ also commutes
with every degree-one element in $\gr_1(PB_4)$. By
Lemma~\ref{lem:bf-stage-center}, its six coefficients must all
be equal. But the formulas for the four operators give the
following coefficients:
\[
\begin{array}{c|cccccc}
\text{pair}&12&13&14&23&24&34\\ \hline
\tau(z_4)&a+b+d&a+b+2d&a+b+c+3d&a+d&a+c+2d&a+c+d.
\end{array}
\]
Equality of the coefficients at $12$ and $13$ gives $d=0$.
Equality at $13$ and $23$ then gives $b=0$. Equality at $24$
and $23$ gives $c=0$. Consequently $\tau=aI=\varepsilon I$,
where $\varepsilon=\pm1$.

Every class in $\gr_k(P)$ is a finite integral sum of Lie
monomials with $k$ degree-one inputs. Applying $\Gamma$ multiplies
each input by $\varepsilon$, and the bilinearity of the bracket
multiplies each such monomial by $\varepsilon^k$.
This proves the formula in all degrees. Multiplication by
$\varepsilon^k$ is invertible, so $\Gamma$ is an automorphism.
\end{proof}

\begin{lem}
\label{lem:bf-normalize-graded}
Let $\Phi:BF\twoheadrightarrow BF$, let $\alpha$ and $\beta$
be as in Lemma~\ref{lem:braided-epi-reduction}, and choose $h$ realizing
$\alpha$ as in Lemma~\ref{lem:bf-spatial-f}. Then
\[
 \Gamma=\mathcal H_h^{-1}\circ\gr(\beta)
\]
is a graded Lie homomorphism with surjective $F$-equivariant
degree-one map $\tau=h_*^{-1}\beta_{\mathrm{ab}}$. Consequently,
\begin{equation}\label{eq:bf-signed-spatial}
 \beta_{\mathrm{ab}}=\varepsilon h_*,\qquad
 \gr_k(\beta)=\varepsilon^k
       \mathcal H_h|_{\gr_k(P)}\quad(k\ge1),
 \qquad\varepsilon\in\{1,-1\}.
\end{equation}
\end{lem}
\begin{proof}
Surjectivity of $\beta$ gives surjectivity of its abelianization.
For $f\in F$, the element $\Phi(s(f))$ projects to $\alpha(f)$,
so write it as $p_f s(\alpha(f))$ with $p_f\in P$.
Apply $\Phi$ to $s(f)ps(f)^{-1}$ for $p\in P$. Inner conjugation
by $p_f$ is trivial on $P_{\mathrm{ab}}$, so the resulting identity
on $M$ is
\[
 \beta_{\mathrm{ab}}(f_*m)
      =\alpha(f)_*\beta_{\mathrm{ab}}(m).
\]
Since $\alpha(f)=hfh^{-1}$, pushforwards satisfy
$h_*f_*h_*^{-1}=\alpha(f)_*$. Hence
\begin{align*}
 \tau(f_*m)
 &=h_*^{-1}\alpha(f)_*\beta_{\mathrm{ab}}(m)\\
 &=f_*h_*^{-1}\beta_{\mathrm{ab}}(m)=f_*\tau(m).
\end{align*}
This proves $F$-equivariance. Surjectivity is retained on composing
with $h_*^{-1}$. Proposition~\ref{prop:induced-graded-map} and
Lemma~\ref{lem:bf-spatial-graded} ensure that $\Gamma$ is a
graded Lie homomorphism. Proposition~\ref{prop:bf-lie-scalar}
now makes its degree-$k$ part $\varepsilon^k I$, proving
\eqref{eq:bf-signed-spatial}.
\end{proof}

The braided Thompson group $BF$ is an extension of the ordered
Thompson group $F$ by pure braids, with a splitting given by
unbraided diagrams. Its order-preserving quotient action has a
larger module commutant than the action of $V$, but its pure braid
Lie brackets constrain the maps that can arise from group
endomorphisms. We can now state the main result of this section.

\begin{thm}\label{thm:bf-hopfian}
The braided Thompson group $BF$ is Hopfian.
\end{thm}
\begin{proof}
Let $\Phi:BF\twoheadrightarrow BF$. By
Lemma~\ref{lem:braided-epi-reduction}, its kernel is contained in $P$
and $\beta=\Phi|_P$ is a surjective endomorphism of $P$.
Lemma~\ref{lem:bf-normalize-graded} shows that each
$\gr_k(\beta)$ is an automorphism.

The group $P$ is residually nilpotent by
Lemma~\ref{lem:p-residual}. By Theorem~\ref{thm:graded-detection}, $\beta$ is injective. Hence $\Phi$ is an automorphism.
\end{proof}

\section{Hopficity of the labeled Thompson group \texorpdfstring{$V(G)$}{V(G)}}
\label{sec:vg-hopficity}

Throughout this section, $G$ is an arbitrary discrete group. The goal of this section is to prove that the labeled Thompson group
$V(G)$ is Hopfian if and only if $G$ is Hopfian. This proves the
conjecture of Wu--Wu--Zhao--Zhou \cite[Introduction]{WuWuZhaoZhou2025}. The main task is to control the restriction
of a surjective endomorphism to the label kernel.

There are many equivalent definitions for the labeled Thompson groups $V(G)$. For example one can use the tree-label-tree diagram to define $V(G)$ which can be find in \cite[Section~1]{WuWuZhaoZhou2025}. Just like in the case for $BV$, one can forget the labels to get a surjective map from $V(G)$ to $V$, the kernel $K$ also has the direct limit structure that $K=\varinjlim_TG^T$. See \cite[Section~1.5]{WuWuZhaoZhou2025} for more details. Note that the exact sequence $1\rightarrow K\rightarrow V(G)\rightarrow V\rightarrow 1$ splits, so we can define $V(G)$ as a semi-direct product.

\begin{defn}
Let $X=\Cantor$ and give group $G$ discrete topology. The group
\[
 K=C(X,G)
\]
consists of all locally constant continuous functions $X\to G$. The operation is pointwise multiplication,
$(fk)(x)=f(x)k(x)$, the identity is the constant function $1$, and
$f^{-1}(x)=f(x)^{-1}$. Every such function is constant on the members of a finite
partition of $X$ into cylinder sets. Using the action of $V$ on the Cantor set, we have the natural action of $V$ on $K$. The labeled Thompson group can be defined as follows:
\begin{equation}
 \label{eq:vg-semidir}
 V(G):=K\rtimes V,
 \qquad (v\cdot f)(x):=f(v^{-1}x).
\end{equation}
\end{defn}
It is not difficult to check that the group of locally constant continuous $G$-functions coincides with the direct limit $\varinjlim_TG^T$ in
\cite[Section~1.5]{WuWuZhaoZhou2025}. We use left actions and hence the multiplication
convention
\[
 (f,v)(k,w)=\bigl(f(v\cdot k),vw\bigr).
\]
The inverse is
$(f,v)^{-1}=(v^{-1}\cdot f^{-1},v^{-1})$. The notation $K\rtimes V$
means the semidirect product with this multiplication and action.
Write $\pi:V(G)\to V$, $\pi(f,v)=v$, for the projection. We identify $K$ with $K\times\{1\}$
and $V$ with $\{1\}\times V$ when no ambiguity can arise.

For $f\in K$, its support is
$\operatorname{supp}(f)=\overline{\{x\in X:f(x)\neq1\}}$, where the bar
denotes closure in $X$. The nonidentity set is already clopen here, so
it equals its closure. For a clopen subset $U\subseteq X$, define
\[
 K_U=\{f\in K:f(x)=1\text{ for all }x\notin U\}.
\]
For $g\in G$, let $g_U\in K_U$ be equal to $g$ on $U$ and to $1$ elsewhere.
In particular, $g_X$ is the constant function with value $g$. We will use
\begin{equation}
 \label{eq:vg-local-identities}
 K=K_U\times K_{X\setminus U},\qquad
 K_U\lhd K,\qquad
 v\cdot K_U=K_{vU},\qquad v\cdot g_U=g_{vU}.
\end{equation}
If $U=U_0\sqcup U_1$ is a disjoint clopen decomposition, then
\begin{equation}
 \label{eq:vg-label-split}
 g_U=g_{U_0}g_{U_1}.
\end{equation}

Similarly to the $BV$ case, the kernel $K$ of $V(G)$ is also non-Hopfian, one can also define the surjective endomorphism of $K$ by deleting the right half of $\mathfrak C$ which is not injective.

Let us first recall the minimality of the $V$ action on the Cantor set.

\begin{lem}
 \label{lem:vg-clopen}
 For each $x\in X$, the stabilizer
 $V_x=\{v\in V:vx=x\}$ acts transitively on the nonempty clopen subsets
 of $X$ that do not contain $x$. Moreover, every $V$-orbit in $X$ is dense.
\end{lem}

The next lemma proves preservation of $K$ under a surjective
endomorphism of $V(G)$.

\begin{lem}
 \label{lem:vg-base-preservation}
 If $\Phi:V(G)\to V(G)$ is surjective, then
 \[
  \Phi(K)=K,\qquad \Phi^{-1}(K)=K.
 \]
 Moreover, $\Phi$ induces an automorphism of $V(G)/K\cong V$.
\end{lem}

\begin{proof}
 Set $q=\pi\circ\Phi:V(G)\to V$. Since $K\lhd V(G)$ and $q$ is surjective,
 $q(K)$ is a normal subgroup of $V$. Simplicity of $V$ gives
 $q(K)=\{1\}$ or $q(K)=V$.

 Suppose that $q(K)=V$. Let $X_0=0X$, $X_1=1X$, and
 $A_i=q(K_{X_i})$. By \eqref{eq:vg-local-identities}, the groups $A_0,A_1$
 are commuting normal subgroups of $V$ and satisfy $A_0A_1=V$.
 Let $s\in V$ be the prefix replacement interchanging $0X$ and $1X$.
 Conjugation by $s$ interchanges $K_{X_0}$ and $K_{X_1}$, so conjugation
 by $q(s)$ interchanges $A_0$ and $A_1$. Since $V$ is simple and
 $A_0A_1=V$, at least one of these normal subgroups is $V$; conjugacy then
 forces both to be $V$. Their commutation would make $V$ abelian, a
 contradiction.

 Thus $q(K)=\{1\}$ and $\Phi(K)\subseteq \ker \pi=K$. The induced endomorphism
 $\overline\Phi:V\to V$ is surjective. As $V$ is nontrivial and simple,
 $\overline\Phi$ is an automorphism. The identity
 $\pi\Phi=\overline\Phi\pi$ and injectivity of $\overline\Phi$ now give
 $\Phi^{-1}(K)=K$. Finally, for $b\in K$, choose $e\in V(G)$ with
 $\Phi(e)=b$. The preimage equality implies $e\in K$, proving
 $\Phi(K)=K$.
\end{proof}

Using Rubin's reconstruction theorem for the quotient $V$-automorphism $\overline\Phi$, we can modify the endomorphism of $\Phi$. After modification, we have a clear description on the induced surjective endomorphism of $K$.

\begin{lem}
 \label{lem:vg-normalize}
 Let $\Phi:V(G)\to V(G)$ be surjective. After composing $\Phi$ with an
 automorphism of its target, one may assume that its induced map on $V$
 is the identity. Under this assumption there are a surjective
 homomorphism $\psi:K\to K$ and a map $c:V\to K$ such that
 \begin{align}
  \psi(v\cdot f)&=c(v)(v\cdot\psi(f))c(v)^{-1}.
  \label{eq:vg-twisted-equivariance}
 \end{align}
\end{lem}

\begin{proof}
 By Lemma~\ref{lem:vg-base-preservation}, the quotient map
 $\overline\Phi$ is an automorphism. By Theorem~\ref{thm:rubin},  There exists
 $h\in \operatorname{Homeo}(X)$ such that
 $\overline\Phi(v)=hvh^{-1}$ for all $v\in V$. For $f\in K$, Precomposing with a homeomorphism of Cantor set, we still get a continuous map in $K$. Thus we can define map as follows:

 \[
  T_h:V(G)\longrightarrow V(G),
  \qquad T_h(f,v)=(f\circ h^{-1},hvh^{-1})
 \]
Next, we say that this is a group homomorphism. First observe that
 \[
  (v\cdot k)\circ h^{-1}
   =(hvh^{-1})\cdot(k\circ h^{-1}),
 \]
for $v\in V$ and $k\in K$. This is because evaluation of either side at $x$ gives $k(v^{-1}h^{-1}x)$.
 Consequently,
 \begin{align*}
  T_h((f,v)(k,w))
   &=\bigl((f\circ h^{-1})((hvh^{-1})\cdot(k\circ h^{-1})),
             hvwh^{-1}\bigr)\\
   &=T_h(f,v)T_h(k,w).
 \end{align*}
 Direct substitution gives $T_{h^{-1}}T_h=T_hT_{h^{-1}}=\mathrm{id}$,
 so $T_h$ is an automorphism. Replace $\Phi$ by
 $T_h^{-1}\circ\Phi$. The replacement has the same kernel $\ker\Phi=\ker (T_h^{-1}\circ\Phi)$ and induces
 the identity on $V$.

 Define $\psi$ and $c$ by
 \[
  \Phi(f,1)=(\psi(f),1),\qquad \Phi(1,v)=(c(v),v).
 \]
 The restriction $\psi$ is a surjective endomorphism by
 Lemma~\ref{lem:vg-base-preservation}. Applying $\Phi$ to $(1,v)(f,1)(1,v)^{-1}=(v\cdot f,1)$ gives
 \eqref{eq:vg-twisted-equivariance}.
\end{proof}

Now for every surjective endomorphism of $V(G)$, after normalization, we can have a surjective endomorphism of the kernel $K$. Next, we say that the induced endomorphism does not enlarge the support of functions in $K$.

\begin{lem}
 \label{lem:vg-locality}
 Suppose that $\psi:K\to K$ is surjective and satisfies
 \eqref{eq:vg-twisted-equivariance} for a map $c:V\to K$. Then
 \[
  \psi(K_U)\subseteq K_U
  \qquad\text{for every clopen }U\subseteq X.
 \]
\end{lem}

\begin{proof}
 Fix $x\in X$ and consider the surjective homomorphism
 \[
  \rho_x:K\longrightarrow G,
  \qquad \rho_x(f)=\psi(f)(x).
 \]
 For a clopen set $U$, put $A_U=\rho_x(K_U)$. Since $K_U\lhd K$ and
 $\rho_x$ is surjective, $A_U\lhd G$. Surjectivity here follows
 from that of $\psi$ and from evaluation of constant functions in $K$. If $v\in V_x$, evaluating
 \eqref{eq:vg-twisted-equivariance} at $x$ gives
 \begin{equation}
  \label{eq:vg-stabilizer-evaluation}
  \rho_x(v\cdot f)=c(v)(x)\rho_x(f)c(v)(x)^{-1}.
 \end{equation}
 In particular, for $v\in V_x$, moving support $U$ to $vU$ does not change the image group:
 \[
  A_{vU}=c(v)(x)A_Uc(v)(x)^{-1}=A_U.
 \]
 Lemma~\ref{lem:vg-clopen} implies that all $A_U$, for nonempty
 clopen $U$ avoiding $x$, equal a single normal subgroup. We denote $J_x\lhd G$ by this normal subgroup of $G$ which only depends on $x\in X$.

 We first show that $J_x\subseteq Z(G)$. Choose two disjoint nonempty
 clopen sets $U,D$ avoiding $x$. The groups $K_U$ and $K_D$ commute,
 hence so do $A_U$ and $A_D$. Both equal $J_x$, so $J_x$ is abelian.
 Furthermore, \eqref{eq:vg-local-identities} and surjectivity of $\rho_x$
 give
 \[
  G=A_UA_{X\setminus U},\qquad [A_U,A_{X\setminus U}]=\{1\}.
 \]
 Since $A_U=J_x$ is abelian, it commutes with each factor in this
 product, proving $J_x\subseteq Z(G)$.

 Fix $g\in G$. For a nonempty clopen $U$ avoiding $x$, the element
 $\rho_x(g_U)$ lies in $J_x$, hence is central. If $U'$ is another such
 clopen set, choose $v\in V_x$ with $vU=U'$. Equations
 \eqref{eq:vg-local-identities} and
 \eqref{eq:vg-stabilizer-evaluation} imply
 \[
  \rho_x(g_{U'})
  =c(v)(x)\rho_x(g_U)c(v)(x)^{-1}
  =\rho_x(g_U).
 \]
 Denote this common value by $t=t(x,g)$. Split $U$ into two nonempty
 clopen sets $U_0,U_1$. The label-splitting identity
 \eqref{eq:vg-label-split} gives
 \[
  t=\rho_x(g_U)
   =\rho_x(g_{U_0})\rho_x(g_{U_1})=t^2,
 \]
 so $t=1$. Thus $\psi(g_U)(x)=1$ whenever $x\notin U$.

 Finally, if $f\in K_U$, choose a finite clopen partition
 $U=U_1\sqcup\cdots\sqcup U_r$ on which $f$ has constant values
 $g_1,\ldots,g_r$. Then
 $f=(g_1)_{U_1}\cdots(g_r)_{U_r}$. For every $x\notin U$, each factor
 has trivial image under $\rho_x$, so $\psi(f)(x)=1$. This is the
 desired inclusion.
\end{proof}

Finally, we want to transfer the surjective endomorphism of $K$ to surjective endomorphisms of $G$ at every Cantor point.

\begin{lem}
 \label{lem:vg-fibers}
 In the normalized situation of Lemma~\ref{lem:vg-normalize}, for every
 $x\in X$ there is a surjective endomorphism $\vartheta_x:G\to G$ such
 that
 \begin{equation}
  \label{eq:vg-pointwise}
  \psi(f)(x)=\vartheta_x(f(x))
  \qquad(f\in K).
 \end{equation}
\end{lem}

\begin{proof}
 Define $\vartheta_x(g)=\psi(g_X)(x)$. Constant functions form a copy
 of $G$ in $K$, so $\vartheta_x$ is a homomorphism. Given $f\in K$,
 let $g=f(x)$ and choose a clopen neighborhood $U$ of $x$ on which
 $f|_U=g$. The function $g_X^{-1}f$ belongs to $K_{X\setminus U}$.
 By Lemma~\ref{lem:vg-locality}, its image under $\psi$ is trivial at
 $x$. Consequently,
 \[
  \psi(f)(x)=\psi(g_X)(x)=\vartheta_x(f(x)).
 \]
 To see surjectivity, fix $a\in G$. Since $\psi$ is surjective, there
 is an $f\in K$ with $\psi(f)=a_X$. Equation
 \eqref{eq:vg-pointwise} then gives $a=\vartheta_x(f(x))$.
\end{proof}

The labeled Thompson group $V(G)=K\rtimes V$ has two complementary
features: arbitrary
$G$-labels can be placed on clopen pieces, while $V$ moves those pieces
by prefix replacements. The quotient-normalization and locality
lemmas show that a surjective endomorphism must respect this local
structure. Although $K$ itself is non-Hopfian for nontrivial $G$, the
maps that extend to the labeled Thompson group $V(G)$ induce
surjective endomorphisms of $G$ at
every Cantor point. This is the additional constraint behind the main
result for the labeled Thompson group $V(G)$.

\begin{thm}
 \label{thm:vg-hopfian}
 For every group $G$, the labeled Thompson group $V(G)$ is
 Hopfian if and only if $G$ is Hopfian.
\end{thm}

\begin{proof}
 Suppose first that $G$ is Hopfian, and let $\Phi:V(G)\to V(G)$ be a
 surjective endomorphism. Normalize the quotient as in
 Lemma~\ref{lem:vg-normalize}; this does not change the kernel.
 Lemma~\ref{lem:vg-fibers} gives surjective endomorphisms
 $\vartheta_x:G\to G$ for all $x\in X$. Hopficity of $G$ makes each
 $\vartheta_x$ injective. 
 
 We consider the associated surjective endomorphism $\psi$ on $K$. If $\psi(f)=1$, where $1$ is the constant map from $X$ to $G$. Using \ref{eq:vg-pointwise}, we have that
 \[
  1=1(x)=\psi(f)(x)=\vartheta_x(f(x))\quad\text{for every }x\in X,
 \]
 and hence $f(x)\in \ker(\vartheta_x)$ for every $x$. The injectivity of $\vartheta_x$ implies that $f(x)=1$ for every $x$. Thus $f$ is the identity of $K$ and $\psi$ is injective.
 Since the induced map on $V(G)/K\cong V$ is injective as well,
 $\ker\Phi\subseteq K$ and
 $\ker\Phi=\ker\psi=\{1\}$. Therefore $\Phi$ is injective and the labeled Thompson group $V(G)$ is Hopfian.

 Conversely, suppose that $G$ admits a surjective noninjective
 endomorphism $\alpha:G\to G$. Applying $\alpha$ to every label defines
 \[
  V(\alpha):V(G)\longrightarrow V(G),
  \qquad (f,v)\longmapsto(\alpha\circ f,v).
 \]
 This is a homomorphism. It is surjective because a locally constant
 function has finite image, and one can choose a preimage under
 $\alpha$ for each of its finitely many values. If
 $1\ne g\in\ker\alpha$, then the nontrivial constant function $g_X$
 lies in $\ker V(\alpha)$. Thus the labeled Thompson group $V(G)$ is not Hopfian.
\end{proof}

\clearpage
\bibliographystyle{alpha}
\bibliography{reference}
\end{document}